\documentclass[11pt]{article}

\usepackage{amsmath}
\usepackage{amsfonts}
\usepackage{amsthm}
\usepackage{amssymb}
\usepackage[pdftex]{graphicx}
\usepackage{epstopdf}
\usepackage{psfrag}
\usepackage[usenames,dvipsnames]{color}

\advance\topmargin by -\headheight
\advance\topmargin by -\headsep
\evensidemargin=\oddsidemargin
\newtheorem{thm}{Theorem}[section]
\newtheorem{Lem}[thm]{Lemma}
\newtheorem{Prop}[thm]{Proposition}
\newtheorem{Cor}[thm]{Corollary}

\newtheorem*{DistanceFormula}{The Lorentzian distance formula}
\newtheorem*{MainResult}{Finiteness of the Lorentzian distance}

\theoremstyle{definition}
\newtheorem{Def}[thm]{Definition}
\newtheorem{Ex}[thm]{Example}

\newcommand{\grad}{\ensuremath{\nabla}}

\newcommand{\R}{\ensuremath{\mathbb{R}}}
\newcommand{\N}{\ensuremath{\mathbb{N}}}

\newcommand{\Future}[1]{\ensuremath{I^+(#1)}}

\newcommand{\Past}[1]{\ensuremath{I^-(#1)}}

\newcommand{\EmptySet}{\ensuremath{\varnothing}}
\newcommand{\esssup}{\ensuremath{{\rm ess\,sup}}}
\newcommand{\essinf}{\ensuremath{{\rm ess\,inf}}}

\title{Generalised time functions and finiteness of the Lorentzian
distance}
\author{Adam Rennie\dag, Ben E. Whale\ddag
\thanks{email: 
\texttt{renniea@uow.edu.au}, \texttt{ben@benwhale.com}
}
\\[3pt]
\dag School of Mathematics and Applied Statistics, University of Wollongong\\
Wollongong, Australia\\[3pt]
\ddag c/o Department of Mathematics and Statistics\\ University of Otago\\
Dunedin, Otago, 9016, NEW ZEALAND
}

\begin{document}
\maketitle

\begin{abstract}
  We show that finiteness of the Lorentzian distance is equivalent to the
  existence of generalised time functions with gradient uniformly
  bounded away from light cones.
  To derive this result we introduce new techniques to construct
  and manipulate achronal sets.
  As a consequence of these techniques we obtain
  a functional description of the Lorentzian distance
  extending the work of Franco and Moretti,
  \cite{FRANCO,Moretti2003Aspects}.
\end{abstract}

%\tableofcontents

\section{Introduction}
\label{sec:intro}

  This paper originated from asking whether Franco and Moretti's formula
  for the Lorentzian distance function $d:M\times M\to [0,\infty]$ could be extended to stably causal 
  manifolds,  \cite{FRANCO,Moretti2003Aspects}. Their proofs
  were valid only in the globally hyperbolic case. The technical difficulties raised 
  by this problem
  led to a consideration of 
  the delicate interplay between the Lorentzian distance function, 
  causality and time functions.

  Ultimately we were led to develop  new techniques for the
  construction of achronal sets, the manipulation of these sets
  and a new class of generalised time function. These new techniques
  allow us to prove our two main results.

  \begin{MainResult}
  \label{thm:main}
    Let $(M,g)$ be a Lorentzian manifold.
    The Lorentzian distance is finite if and only if there exists
    a function $f:M\to\R$, strictly monotonically increasing on timelike curves, whose gradient
    exists almost everywhere and is such that
    $\esssup\, g(\grad f,\grad f)\leq -1$.
  \end{MainResult}

  %It is not immediately apparent that being strictly monotonic on all
  %causal curves implies that the gradient exists, although this is the case.
  %Appendix \ref{sec_diff_tfs} provides the details.

 % Our new techniques also allow us to prove our other main result, 

  \begin{DistanceFormula}
  \label{thm:distanceformula}
    Let $(M,g)$ have finite Lorentzian distance.
    Then for all $p,\,q\in M$
    \begin{align}
    d(p,q)=\inf\left\{\max\{f(q)-f(p),0\}:\ 
        f:M\to\R,\ f\ {\rm future\ directed},\ \esssup\,g(\nabla f,\nabla f)\leq -1\right\}.
     % d(p,q)=\inf\left\{|f(q)-f(p)|:\ 
        %f:M\to\R,\ \esssup\,g(\nabla f,\nabla f)\leq -1\right\}.
        \label{eq:distance-formula-2}
    \end{align}
  \end{DistanceFormula}

  We will refer to equality \eqref{eq:distance-formula-2} as the distance formula
  below.

  Franco and Moretti had as their initial motivation the extension of Connes'
  formula for the Riemannian distance to Lorentzian manifolds,
  \cite{Connes1989Compact}, and we note that early investigations and counter-examples appear in
  \cite{PZ}.  The tools of noncommutative geometry have thus far
  not been seriously extended past the globally hyperbolic setting, and we hope
  that our results stimulate further work on this topic.

  The paper is organised as follows. Section \ref{sec:back} summarises those
  ideas from Lorentzian geometry that we require, and sets notation. In
  addition we review, and mildly extend, the results of Franco and Moretti. We also
  prove a `reverse Lipschitz' characterisation of our generalised time functions in 
  Proposition \ref{cor:demarcated-thm}, which is essential for applications to Connes-type
  formulae for the distance.
  
  In brief, the idea of our proof is as follows. Let $S\subset M$ be an achronal set in the
  Lorentzian manifold $(M,g)$. Then if $M=I^+(S)\cup S\cup I^-(S)$, we can try to define
  a function $f(x)=d(S,x)=\sup_{s\in S}d(s,x)$ 
  when $x$ is in the future of $S$, and similarly for other cases.
  The chief difficulty with this definition is the finiteness of $f$, even when the Lorentzian distance
  function $d$ only takes finite values. Much of the difficulty is in finding a suitable set
  $S\subset M$ with which to define $f$.

  Section \ref{sec:guts} contains the technical advances, and is divided into
  three subsections. The first shows that if the Lorentzian distance
  is finite then it is possible to choose an achronal subset of the manifold
  that `bounds' any divergent behaviour of the metric.
  The second proves that, under mild assumptions on $M$, and starting from a suitable achronal set,
  there exists
  an achronal surface which divides the manifold into the future of the set,
  the surface itself and the past of the surface.
  This is a refinement of a construction
  of Penrose, \cite[Proposition 3.15]{Penrose1987Techniques}.
  The third section shows how, starting from such a `bounding' achronal set, to 
  construct a new achronal set. This produces a new achronal set $S$ which separates 
  the manifold $M$ into the future of $S$,
  $S$ itself and the past of $S$. The advantage of this new set is 
  that we can define a generalised time function
  by taking the Lorentzian distance of a point to $S$, and this function takes finite values.

  Finally, Section \ref{sec:main-results} presents the proofs of our two main
  results.
  
  The Appendix provides the details on the regularity of our generalised time
  functions.  A similar concept, also called generalised time functions, 
  has appeared previously, \cite{Hawking-cosmic}.  Our generalised time functions have poor regularity,
  but in the Appendix we prove that they are continuous almost everywhere, and
  do have all directional derivatives, and so gradient, existing almost
  everywhere. 

  {\bf Acknowledgements:} We thank Koenraad van den Dungen for pointing out
  some errors in an earlier version.
  The first author acknowledges the support of the Australian Research Council.
  The second author would like to thank the relativity group at the Department 
  of Mathematics and Statistics at the
  University of Otago for useful comments during review and acknowledges the
  support of the Royal Society of New Zealand's Marsden fund. The authors thank the anonymous referee
  for their comments which have improved the paper, and for bringing the article \cite{PZ}
  to our attention.

\section{Background definitions, notation and results}
\label{sec:back}
    In the following $(M,g)$ will always be a $C^\infty$,
    time orientable, path-connected, Lorentzian manifold $M$ of dimension 
    $n+1\geq 2$
    equipped with a Lorentzian metric $g$ with signature $(-1,1,\ldots,1)$. 
    We let $T$ denote the 
    vector field defining the time orientation.
    The non-time orientable case can be studied 
    via Lorentzian covering manifolds, \cite[p 181]{Hawking1975Large}. 
    Here and below the measure is always the Lebesgue measure 
    arising from $\sqrt{-\det{g}}$.
    Throughout the rest of the paper, unless otherwise noted, we shall
    use the notation as given in \cite{BeemEhlichEasley1996}. In particular,
    for any
    $U\subset M$, $I^\pm(U)=\bigcup_{x\in U}I^\pm(x)$
    and that for any $x\in M$,
    $I^+(x)=\{y\in M: \ d(x,y)>0\}$.

    A curve $\gamma$ is a $C^0$, piecewise $C^{1}$, function 
    from an interval $I\subset\R$ 
    into $M$ so that the tangent vector $\gamma'=\gamma_*
    (\partial_t)$ is almost everywhere (a.e.) non-zero. For $x,\,y\in M$
    we let $\Omega_{x,y}$ denote the set of future-directed causal
    curves from $x$ to $y$. Thus $\gamma\in \Omega_{x,y}$ satisfies 
    $g(\gamma',\gamma')\leq 0$ (causal) everywhere it exists 
    and $g(T,\gamma')<0$ (future-directed).

    By a standard abuse of notation, we sometimes treat $\gamma$ as a set 
    rather than a curve. Thus
    $x,\,y\in\gamma$ means $x,\,y\in\gamma(I)$, $\gamma\subset U$ 
    means $\gamma(I)\subset U$, and so on.
    Given a causal curve $\gamma:[a,b]\to M$, the length
    of $\gamma$, denoted $L(\gamma)$ is defined by
    \[
	    L(\gamma)=\int_a^b\sqrt{-g(\gamma',\gamma')(t)}dt.
    \]

    \begin{Def}[{\cite[Chapter 4]{BeemEhlichEasley1996}}]
      \label{def:lorentzian-distance}
      Let $(M,g)$ be a Lorentzian manifold. The Lorentzian distance
      $d:M\times M\to\R$ is given by
      \[
        d(p,q):=\left\{\begin{aligned}
          &\sup_{\gamma\in\Omega_{p,q}}L(\gamma) & \quad\quad 
          &\Omega_{p,q}\neq\EmptySet\\
          &\hspace{5pt}0  & \quad\quad &
          \Omega_{p,q}=\EmptySet.
        \end{aligned}\right.
      \]
    \end{Def}
    The Lorentzian distance is always lower semi-continuous,
    \cite[Lemma 4.4]{BeemEhlichEasley1996}.
    
    We make use of the reverse triangle inequality for the Lorentzian 
    distance, \cite[page 140]{BeemEhlichEasley1996}:
    if $x\in M$, $y\in I^+(x)$ and $z\in I^+(y)$ then
    $d(x,z)\geq d(x,y) + d(y,z)$. If for all $x,y\in M$, %and $y\in I^+(x)$
    $d(x,y)<\infty$ then we say that the Lorentzian distance is finite,
    or that $M$ has finite Lorentzian distance.

    \begin{Def}\label{def:d(S,)}
      Let $S\subset M$ be
      a subset of $M$.
      We define the functions $d(S,\cdot):M\to\R\cup\{\infty\}$ and 
      $d(\cdot, S):M\to\R\cup\{\infty\}$
      by
      $
        {d}(S,x)=\sup\{{d}(s,x):s\in S\}
      $ 
      and
      $
        {d}(x,S)=\sup\{{d}(x,s):s\in S\}.
      $
    \end{Def}

    These functions satisfy a version of the reverse triangle inequality.

    \begin{Lem}\label{lem.inequalities}
      Let $x\in M$, $y\in I^+(x)$ and $S\subset M$. Then:
      \begin{enumerate}
        \item $x\in I^+(S)$ implies that ${d}(S,y)\geq {d}(S,x)+{d}(x,y)$;
      %  \item $y\in I^+(S)$ and $x\in I^-(S)$ implies that 
         % ${d}(x,y)\geq {d}(x,S)+{d}(S,y)$;
        \item $y\in I^-(S)$ implies that ${d}(x,S)\geq {d}(x,y)+{d}(y,S)$.
      \end{enumerate}
    \end{Lem}
    \begin{proof}
      In each case, the reverse
      triangle inequality implies that: 
      \begin{enumerate}
        \item $d(z,y)\geq d(z,x)+d(x,y)$ when $z\in S\cap I^-(x)$;
       % \item $d(x,y)\geq d(x,z)+d(z,y)$ when $z\in S\cap I^+(x)\cap I^-(y)$;
        \item $d(x,z)\geq d(x,y)+d(y,z)$ when $z\in S\cap I^+(y)$.
      \end{enumerate}
      Taking the supremum over these inequalities with respect to $z$ proves
      the result.
    \end{proof}

    \begin{Def}
    \label{def:gen-time-fctn}
      A function $f:M\to\R$ such that for all timelike curves from $x$ to $y$
      the function
      $f\circ\gamma$ is strictly monotonically increasing is called
      a future-directed generalised time function. A past-directed
      generalised time function $f$ is a function so that $-f$ is a future-directed
      generalised time function.
%      If $f:M\to\R$ is a continuous (future or past-directed) 
%      generalised time function 
%      then we say that $f$ is a (future or past-directed) time function.
    \end{Def}

    It is worth noting that our definition of a generalised time function
    is slightly more general than that used in the literature, 
    \cite[Definition 3.23]{BeemEhlichEasley1996} or
    \cite[Definition 3.48]{Minguzzi2008Causal}, as we do not
    require our generalised time functions to be strictly monotonically
    increasing along null curves.

    We show, in Appendix \ref{sec_diff_tfs},
    that if $f$ is monotonic on all timelike curves, as are
    generalised time functions, then all directional 
    derivatives of $f$ exist a.e. This is an application of the
    well-known differentiability a.e.  of real-valued monotonic functions on an
    internal, \cite[Theorem 9.3.1]{Haaser1991Real}.
    Moreover, \cite[Chapter 5, Theorem 2]{Royden}, for any function $f:M\to \R$ 
    and any  curve $\gamma:[0,1]\to M$
    such that $f\circ \gamma$ is (not necessarily strictly)
    monotonically increasing, the derivative of $f\circ\gamma$ exists a.e., 
    is integrable, and we have the
    inequality
    \begin{equation}
    \int_0^1\frac{d}{dt}\,(f\circ\gamma)(t)dt\leq f(1)-f(0).
    \label{eq:useful}
    \end{equation}

    In what follows %it will become clear that 
    we will be interested in
    generalised time functions, $f:M\to\R$, so that we have the inequality
    $\esssup_M g(\nabla f, \nabla f)\leq -1$. This condition
    ensures that wherever $\nabla f$ exists, and it must exist a.e.,
    it is timelike, as we show in Proposition \ref{cor:demarcated-thm}. 
    As the following example shows, 
    this is unfortunately not enough to ensure that
    $f$ is strictly monotonically increasing along all causal curves.

    \begin{Ex}
      Let $M=(-1,1)\times(0,1)$ 
      considered as a sub-manifold of $2$-dimensional Minkowski space.
      Let $S=M\cap\partial\Future{(0,0)}$ be the portion of the future
      null cone of the origin that lies in $M$.
      By construction $S$ is achronal and by definition of $M$,
      $M=\Future{S}\cup S\cup\Past{S}$.
      The Lorentzian distance of $M$ is bounded above by $1$ and
      hence is finite.
      Thus the function defined by
      $$
        f(x)=\left\{\begin{array}{ll}  {d}(S,x) & \text{if}\ x\in I^+(S)\\
        0 & \text{if}\ x\in S\\
        -{d}(x,S) & \text{if}\ x\in I^-(S)\end{array}\right..
      $$
      is well defined and a generalised
      time function.
      Proposition \ref{lem:thefunction} shows that
      $f$ satisfies the gradient bound
      $\esssup_M g(\nabla f, \nabla f)\leq -1$. However the level surface
      $S=f^{-1}(0)$ is not acausal, and contains null geodesics on which
      $f$ is constant. The function $f$ is therefore not
      strictly monotonically increasing along all causal curves.
    \end{Ex}

    Lorentzian manifolds, $(M,g)$, can be classified into a causal
    hierarchy. Of that hierarchy we
    shall need the following definitions: 
    \begin{itemize}
      \item {\em stably causal} if there exists a 
            continuous function $f:M\to\R$ 
            that is strictly monotonically increasing on all causal curves,
      \item {\em causally simple} if it is causal and
            $J^\pm(x)$ is closed for all $x\in M$;
      \item {\em globally hyperbolic} if and only if
              it is causal and, for all $x,y\in M$, the 
              intersection $J^+(x)\cap J^-(y)$
              is compact. This is equivalent to $M$ being 
              isometric to the product $\R\times N$. See
              \cite[Section 3.11.3 and Theorem 3.78]{Minguzzi2008Causal}
              for a review of Bernal and Sanchez's work
              on this, 
              \cite{Bernal2003OnSmooth,Bernal2005Smoothness,Bernal2006Further}.
    \end{itemize}
    Global hyperbolicity implies causal simplicity 
    which implies stable 
    causality. See \cite{Minguzzi2008Causal}
    for further details and examples.

    The following example of a non-continuous generalised time function
    with timelike gradient a.e. everywhere demonstrates that the lack of continuity
    in our definition of generalised time function
    can have a serious impact on the relationship between time functions
    and stable causality.

    \begin{Ex}
     Let 
     \[
       T=[-\pi, \pi]\times\R\setminus\left(
         \left\{\left(\frac{\pi}{4}, x\right): x\leq \frac{\pi}{4}\right\}\bigcup
         \left\{\left(-\frac{\pi}{4}, x\right): x \geq -\frac{\pi}{4}\right\}\right),
     \]
     with coordinates $t\in[-\pi,\pi]$ and
     $s\in \R$. A diagram representing this manifold can be found
     in \cite[Figure 7]{Minguzzi2008Causal} and additional discussion of
     this example can be found in \cite[Page 193]{Hawking1975Large}
     and
     \cite[Figure 3.4]{BeemEhlichEasley1996}.
     Let $(t,s),(\tau,\sigma)\in T$ and define an equivalence relation,
     $\sim$, on $T$ by $(t,s)\sim(\tau,\sigma)$ if and only if
     $s=\sigma$ and $t=-\tau=\pm\pi$.
     Let $M=T/\!\sim$. Topologically $M$ is $S^1\times\R$ with
     two half lines removed. 
     Define a metric $g$ on $M$ by pushing the metric
     $g=-dt^2 + ds^2$ on $T$ onto $M$ via the induced map from $T$ to $M$.
     
     We claim that $(M,g)$ is not stably causal. 
     Indeed, consider the point $(0,0)$. For any metric with slightly wider 
     lightcones, there will exist $\epsilon>0$ such that the point
     $\left(\frac{\pi}{4}, \frac{\pi}{4}+\epsilon\right)$
     is in the future of $(0,0)$ and
     $\left(-\frac{\pi}{4},-\frac{\pi}{4}-\epsilon\right)$
     is in the past of $(0,0)$. By the definition of $M$, the point
     $\left(-\frac{\pi}{4},-\frac{\pi}{4}-\epsilon\right)$
     is in the future of 
     $\left(\frac{\pi}{4}, \frac{\pi}{4}+\epsilon\right)$,
     and hence there will exist a closed timelike curve for any metric with slightly wider
     lightcones.

     Let
     \begin{align*}
       A&=\left\{\left(t,s\right):t>\frac{\pi}{4}, s\in\R\right\}
         \bigcup\left\{\left(t,s\right):t > -\frac{\pi}{4}, s> t\right\}\\
      B&=\left\{\left(t,s\right): t < \frac{\pi}{4}, s < t\right\}\bigcup
         \left\{\left(t,s\right): t < -\frac{\pi}{4}, s\in\R\right\}
     \end{align*}
     and define $f:M\to\R$ by
     \[
       f(t,s) = \left\{\begin{aligned}
         t && (t,s)\in A\\
         t + 2\pi && (t,s)\in B
      \end{aligned}\right.
     \]
     It can easily be checked that $\grad f=\partial t$ wherever it exists and that
     $f$ is a generalised time function despite $M$ failing to
     be stably causal.
    \end{Ex}

    \begin{Def}
    A set $F$ is a future set ($P$ is a past set) if $F=I^+(F)$ ($P=I^-(P)$), 
    \cite[Definition 3.1]{Penrose1987Techniques}.
    A set $S$ is achronal if $S\cap I^+(S)=\varnothing$, 
    \cite[Definition 3.11]{Penrose1987Techniques}, or equivalently $S\cap I^-(S)=\varnothing$. 
    A set $S$
    is an achronal surface (or sometimes also called an achronal boundary) 
    if $S=\partial F$ ($S=\partial P$) where $F$ is a future set 
    ($P$ is a past set), 
    \cite[Definition 3.13 and Proposition 3.14]{Penrose1987Techniques}. 
    Achronal surfaces
    are achronal sets, \cite[Definition 3.13]{Penrose1987Techniques}.
    \end{Def}

    The following result, which is a paraphrase of a result by Penrose,
    highlights the importance of achronal
    surfaces.

  \begin{Prop}[{\cite[Proposition 3.15]{Penrose1987Techniques}}]
    \label{prop.penrose3.15}
    Let $(M,g)$ be a Lorentzian manifold.
    If $S\neq\varnothing$ is an achronal surface then there is a unique
    future set $F$ and a unique past set $P$ so that $F,P,S$ are disjoint,
    $M=F\cup S\cup P$ and $S=\partial F=\partial P$. Furthermore,
    any timelike curve from $P$ to $F$ intersects $S$ in a unique point.
  \end{Prop}

  Proposition \ref{prop.penrose3.15} will play a pivotal role below. To simplify
  its application we use the following two results.

  \begin{Cor}\label{cor.penrose3.15}
    Let $(M,g)$ be a Lorentzian manifold and
    let $S\subset M$ be an achronal surface.
    If $S=\partial I^+(S)$ then 
    $F=I^+(S)$ and $P=M\setminus\left(I^+(S)\cup S\right)$
    are the unique future and past sets given by Proposition
    \ref{prop.penrose3.15}. In particular any timelike curve from $P$
    to $F$ intersects $S$ in a unique point.
  \end{Cor}
  \begin{proof}
    By assumption $S=\partial F=\partial P$ and $M=F\cup S\cup P$. The result
    now
    follows directly from Proposition \ref{prop.penrose3.15}.
  \end{proof}

  \begin{Lem}\label{lem.penrose3.15}
    Let $(M,g)$ be a Lorentzian manifold.
    If $A\subset M$ then $\partial I^+(\partial I^+(A))=\partial I^+(A)$.
  \end{Lem}
  \begin{proof}
    Since $\partial I^+(A)$ is achronal,
    $\partial I^+(A)\subset \partial I^+(\partial I^+(A))$.
    Let $a\in \partial I^+(\partial I^+(A))$. Then 
    $I^+(a)\subset I^+(\partial I^+(A))$. So
    for all $x\in I^+(a)$ there exists $b\in\partial I^+(A)$
    so that $x\in I^+(b)$. This implies that
    $I^+(a)\subset I^+(A)$ and hence
    $a\in\overline{I^+(A)}$. Since $a\in \partial I^+(\partial I^+(A))$
    we know that $a\in\partial I^+(A)$.
    Thus 
    $\partial I^+(\partial I^+(A))=\partial I^+(A)$ as required.
  \end{proof}

  Note that in Corollary \ref{cor.penrose3.15} the set $P$ is not necessarily
  equal to $I^-(S)$.
  A core part of this paper is the
  construction of an achronal surface, $S$, so that the unique future and
  past sets given by Proposition \ref{prop.penrose3.15} are $F=I^+(S)$
  and $P=I^-(S)$. This allows us to assume that every point in $M$
  is either in $S$ or is connected to $S$ via a timelike curve, a fact that
  we will exploit to define the needed functions.
    
\subsection{Overview of Franco's result}\label{sec:franco}
    We briefly reprise the key arguments used by Franco in \cite{FRANCO} to 
    obtain his main result, stated here as Theorem \ref{thm:Franco}, 
    and point out that
    these arguments are broadly similar to those used 
    by Moretti \cite[Theorem 2.2]{Moretti2003Aspects}. 
    
    The version of these results we present 
    represents only a small generalisation, but 
    it seems worthwhile to repeat the arguments, as they show clearly where
    several constraints come from.

    \begin{Lem}\label{Lem bound}
      Let $(M,g)$ be a Lorentzian manifold, $x,\,y\in M$ and
      $\gamma\in\Omega_{x,y}$. 
      If $f$ is monotonic on every timelike curve then
      $
        |f(y)-f(x)|\geq 
          l(\gamma)\,\underset{\gamma}{\essinf}\sqrt{-g(\grad f,\grad f)}.
      $
    \end{Lem}

    \begin{proof}
      By Lemma \ref{lem:gentf_grad_exists} the vector field $\grad f$ exists 
      a.e., see Definition \ref{def:gradfforfLLIP}. Lemma 
      \ref{lem:monotonicImpliesPastDirectedAndCausal} implies
      that $\grad f$ is causal.
      Assume that 
      $\nabla f$ is past-directed.
      Let $\gamma:[0,1]\to M\in\Omega_{x,y}$. 

      We calculate, using Lemmas 
      \ref{lem:monotonic_increas_diff_time_curves} and 
      \ref{lem:gentf_grad_exists} and 
      Definition \ref{def:gradfforfLLIP},  as well as Equation \eqref{eq:useful} that
      \begin{align*}
        f(y)-f(x)&\geq\int_0^1\frac{d}{dt}f(\gamma(t))dt
          =\int_0^1df(\gamma')dt
          =\int_0^1 g(\grad f,\gamma')dt
          =\int_0^1|g(\grad f,\gamma')|dt,
      \end{align*}
      since $\gamma$ is future-directed and $\grad f$ is past-directed.
      By \cite[Proposition 5.30]{ONeill1983SemiRiemannian} if $\gamma'$
      and $\grad f$ are both time-like 
      the reverse Cauchy 
      inequality
      holds, 
      \[
        \lvert g(\grad f, \gamma')\rvert\geq \sqrt{-g(\grad f, \grad f)}
          \sqrt{-g(\gamma',\gamma')}.
      \]
      If $\grad f$ or $\gamma'$ is null then it is clear that this
      inequality continues to hold.
      Hence
      \begin{align*}
        f(y)-f(x)&\geq\int_0^1|g(\grad f,\gamma')|dt\notag 
          \geq \int_0^1\sqrt{-g(\grad f,\grad f)}
            \sqrt{-g(\gamma',\gamma')}dt\notag 
        \geq l(\gamma)\,\underset{\gamma}{\essinf}\sqrt{-g(\grad f,\grad f)}.
      \end{align*}
      
      In the case that $f$ is past-directed,
      $\grad f$ is future-directed. Thus $-f$ has a 
      past-directed gradient, whence
      \[
      	f(x)-f(y)\geq 
          l(\gamma)\,\underset{\gamma}{\essinf}\sqrt{-g(\grad f,\grad f)}\ \ 
          {\rm and\ so}\ \ 
          |f(y)-f(x)|\geq 
          l(\gamma)\,\underset{\gamma}{\essinf}\sqrt{-g(\grad f,\grad f)},
      \]
      as required.

      If $\gamma$ is causal, but neither timelike nor null, we can divide
      $\gamma$ into null and timelike segments. Since $\gamma$ is piecewise
      $C^1$ the intermediate value theorem shows that each timelike segment
      is an open interval. The result now follows by applying the arguments
      given above to each segment.
    \end{proof}

%    \begin{Cor}\label{cor:comparing_definitions_of_generalised_time_functions}
%      With the assumptions of Lemma \ref{Lem bound},
%      If $f$ is future-directed for all causal curves $\gamma$ and
%      $\esssup_\gamma g(\nabla f, \nabla f)\leq -1$ then
%      $f$ is strictly monotonically increasing along all causal curves.
%    \end{Cor}
%    \begin{proof}
%      By Lemma \ref{Lem bound}, and assuming that $f$ is future-directed,
%      $
%        f(y)-f(x)\geq 
%          l(\gamma)\,\underset{\gamma}{\essinf}\sqrt{-g(\grad f,\grad f)}.
%      $
%      Hence as ${\essinf}\sqrt{-g(\grad f,\grad f)}\geq 1$ by assumption
%      we have that
%      $f(y)-f(x)>0$ which gives the result.
%    \end{proof}

    \begin{Cor}\label{coro bound}
      With the assumptions of Lemma \ref{Lem bound} and assuming 
      that $\essinf_{M}\sqrt{-g(\grad f,\grad f)}>0$
      and $d(x,y)<\infty$ 
       we have 
      \[
        |f(y)-f(x)|\geq 
          d(x,y)\,\underset{M}{\essinf}\sqrt{-g(\grad f,\grad f)}.
      \]
    \end{Cor}

    \begin{proof}
      Since
      $
	      {\essinf}_{\gamma}\sqrt{-g(\grad f,\grad f)}\geq 
        {\essinf}_{M}\sqrt{-g(\grad f,\grad f)}
      $
      we have, from Lemma \ref{Lem bound},
      \[
	      |f(y)-f(x)|\geq  l(\gamma)\,
          \underset{M}{\essinf}\sqrt{-g(\grad f,\grad f)}.
      \]
      Taking the supremum over curves in
      $\Omega_{x,y}$ gives 
      $|f(y)-f(x)|\geq 
        d(x,y)\,{\essinf}_M\sqrt{-g(\grad f,\grad f)}$.
    \end{proof}

%     In particular, if
%     ${\essinf}_M\sqrt{-g(\grad f,\grad f)}=0$ then Lemma \ref{Lem bound} does not give an upper bound for $d(p,q)$.

    In order to obtain his functional description of the Lorentzian distance, 
    Franco proves the following, \cite[Lemma 5]{FRANCO}. 

    \begin{Lem}\label{GH bounded timefunctions exist}
      Let $(M,g)$ be globally hyperbolic, $x,\,y\in M$ and $\epsilon>0$. 
      Then there exists a future-directed time function $f$
      so that ${\essinf}_{M}\sqrt{-g(\grad f,\grad f)}\geq 1$ and
      $|f(y)-f(x)- d(x,y)|\leq \epsilon$.
    \end{Lem}

    Global hyperbolicity is essential for Franco's construction
    of this time function. In particular he exploits the existence of
    Cauchy surfaces as well as the necessary finiteness and continuity of
    the Lorentzian distance.
    With this in hand, Franco is able to prove his main result.
  
    \begin{thm}{\cite[Theorem 1]{FRANCO}}\label{thm:Franco}
    Let $(M,g)$ be a globally hyperbolic manifold then
    \begin{equation*}
      d(x,y)=\inf\{\max\{f(y)-f(x),0\}:\  f\in C(M,\R),\  \esssup\, 
        g(\grad f,\grad f)\leq -1,\  \grad f\ \mbox{is past-directed}\}.
    \label{eq:distance-formula}
    \end{equation*}
    \end{thm}
    
    %We note that Franco's proof relies
    %on the continuity and differentiability almost everywhere 
    %of the Lorentzian distance 
    %in 
    %globally hyperbolic manifolds.
    Since Moretti uses different differentiability 
    conditions, his analogue of 
    this result,  \cite[Theorem 2.2]{Moretti2003Aspects},
    is superficially different but has essentially 
    the same conclusion and proof.
    
    The results above imply the following about those
    situations where the Lorentzian distance becomes infinite.

    \begin{Prop}\label{lem inf dis}
      Let $(M,g)$ be a Lorentzian manifold,  
      $x,\,y\in M$ and suppose that $f:M\to\R$  is
      monotonic on every timelike curve. Suppose further that  there
      exist $\{\gamma_i:i\in\N\}\subset\Omega_{x,y}$ so that 
      $l(\gamma_i)\to\infty$, as $i\to\infty$, i.e. $d(x,y)=\infty$.
      Then
      $\lim_{i\to\infty}\essinf_{\gamma_i}\sqrt{-g(\grad f,\grad f)}=0$. 
    \end{Prop}

    \begin{proof}
      This follows from Lemma \ref{Lem bound}, since for all $i$ we have 
      $|f(y)-f(x)|\geq 
        l(\gamma_i)\,\underset{\gamma_i}{\essinf}\sqrt{-g(\grad f,\grad f)}$.
    \end{proof}

    \begin{Cor}
    \label{cor:gen--fin}
      Let $(M,g)$ be a Lorentzian manifold. If there exists 
      a function that is monotonic on every time-like curve
      so that $\essinf_M\sqrt{-g(\grad f, \grad f)}>0$
      then $M$ has finite Lorentzian distance.
    \end{Cor}
    \begin{proof}
      This is implied by the contrapositive of Proposition \ref{lem inf dis}.
    \end{proof}

The behaviour described in Proposition \ref{lem inf dis} can 
occur in otherwise innocuous situations, and has to be taken
into account for our construction.
The following example of a causally simple non-globally 
hyperbolic spacetime with
$x,\,y\in M$ so that $d(x,y)=\infty$ is taken 
from \cite[Remark 3.66]{Minguzzi2008Causal}, and shows how our main construction, 
presented in subsection \ref{subsec:surf-hat} below, fails 
when the Lorentzian distance is not finite. Similar examples with finite Lorentzian distance motivate
the constructions of the next section.

  \begin{Ex}\label{ex.causallysimple.infdistance}
    Let $M=\{(x,y)\in\R^2: 2 \lvert y\rvert> x\ \text{and}\ x>-1\}$ 
    with metric
    $
      ds^2 = \frac{1}{x^2+y^2}\left(dx^2-dy^2\right).
    $
		This is a non-globally hyperbolic, causally simple spacetime, 
		\cite[Figure 10]{Minguzzi2008Causal}.
		As a consequence there exist analytic time functions on $M$.
		A specific example is $h(x,y)=y$ whose gradient is
		$
		  \grad h = -(y^2+x^2)\partial_y.
		$
%	  This is a past-directed timelike vector field, relative to the 
%	  timelike vector field $\partial_y$, and therefore $h$ is a 
%	  future-directed time function.
		
		By definition,  for all $(x,y)\in M$
		the surface $\partial I^+((x,y))$ is an achronal surface, which further satisfies
		$M=I^+(\partial I^+((x,y)))\cup 
      \partial I^+((x,y))\cup I^-(\partial I^+((x,y)))$.
		Hence for any $(x,y)\in M$ and letting $S=\partial I^+((x,y))$,
		 we
		can try to construct a function $f:M\to\R$ by the definition
		\begin{align*}
      f((u,v))=\left\{\begin{aligned}
        {d}(S,(u,v)) && \text{if}\ (u,v)\in I^+(S)\\
        0 && \text{if}\ (u,v)\in S\\
        -{d}((u,v),S) && \text{if}\ (u,v)\in I^-(S)
      \end{aligned}\right..
    \end{align*}
    Depending on the choice of $(x,y)$ we have
    three cases. To present these cases we consider $M$ as a submanifold
    of $\R^2$ and in the following statements closures are taken in $\R^2$.
    The three cases are: 
    \begin{enumerate}
      \item $(0,0)\in\overline{S}$,
      \item $(0,0)\not\in\overline{S}$ and $(0,0)\in \overline{I^-(S)}$,
      \item $(0,0)\not\in\overline{S}$ and $(0,0)\in \overline{I^+(S)}$.
    \end{enumerate}
    For the sake of this example we assume that the last case holds. Note that
    arguments similar to those given below will hold in the other two cases.
    We denote the set
    $
      \{(u,v)\in M: \lvert u\rvert<v\}
    $
    by $I^+((0,0))$. This is an abuse of notation since 
    $(0,0)\not\in M$.

    We now show that for all $(u,v)\in I^+((0,0))$,
    $f(u,v)=\infty$.  
    Let $w>0$ and
    let $\gamma_w:[0,1]\to\R$ be the curve given
    by
    $
      \gamma_w(\tau)=\left(0, w(1-\tau)\right).
    $
    This is a past-directed timelike curve from $(0,w)$ to $(0,0)$,
    and
    \[
      g(\gamma'_w,\gamma'_w)=\frac{1}{(1-\tau)^2}
    \]
    A short calculation shows that 
    $L(\gamma_w)=\infty$ for all $w>0$. 
    Since $w$ was arbitrary we can choose $w$ so that $(0,w)\in\Past{(u,v)}$.
    Since $(0,0)\in\overline{I^+(S)}$ and $(0,0)\not\in \overline{S}$
    we know that for all $\tau\in[0,1)$,
    \[
      \left(0,w(1-\tau)\right)\in I^+(S).
    \]
    Thus, from Lemma \ref{lem.inequalities},
    \begin{align*}
      f(u,v)&=d(S,(u,v))\geq d\left(S,\left(0,w\right)\right)
          +d\left(\left(0, w\right),(u,v)\right)\\
        &=\sup_\tau\left\{d\left(S,\left(0,w(1-\tau)\right)\right)
          +d\left(\left(0,w(1-\tau)\right), \left(0,w\right)\right)\right\}
          +d\left((0,w), (u,v)\right)\\
        &=\infty
    \end{align*}
    as claimed.

%    
%    
%    We now show that for all $(u,v)\in I^+((0,0))$,
%    $f(u,v)=\infty$.  
%    For each $i\in\N$, $i>1$ let $\gamma_i(\tau):[0,1]\to\R$ be the curve given
%    by
%    \[
%      \gamma_i(\tau)=\left(\frac{u-v}{i},0\right)
%        +\tau\left((u,v)-\left(\frac{u-v}{i},0\right)\right).
%    \]
%    Hence
%    \[
%      g(\gamma_i',\gamma_i')
%      =-\ \frac{i^2v^2-\left(u(i-1)+v\right)^2}{i^2\tau^2v^2+\left(u-v+\tau(u(i-1)+v)\right)^2}.
%    \]
%    Since $i>1$ and $\lvert u\rvert<v$ we conclude that $\gamma$ is
%    a time-like curve. We calculate that 
%    $L(\gamma_i)\to\infty$ as $i\to\infty$. 
%    Since $(0,0)\in\overline{I^+(S)}$ and $(0,0)\not\in \overline{S}$
%    we know that there exists $N\in\N$ so that for all $i>N$,
%    \[
%      \left(\frac{u-v}{i},0\right)\in I^+(S).
%    \]
%    Thus, from Lemma \ref{lem.inequalities},
%    \[
%      f(u,v)=d(S,(u,v))\geq 
%      d\left(S,\left(\frac{u-v}{i},0\right)\right)+d\left(\left(\frac{u-v}{i},0\right),(u,v)\right)\to\infty,
%    \]
%    as claimed. 
    
    That is, despite the existence of smooth finite valued
    time functions, the construction we give in Proposition 
    \ref{lem:thefunction}, for this surface,
    produces a function which takes infinite values. We claim that this is the 
    case for all choices of 
    $S=\partial I^+((x,y))$ (in the case $(0,0)\not\in\overline{S}$ and $(0,0)\in \overline{I^-(S)}$
    the function $f$ will take the value $-\infty$). 
    
    Let $U\subset M$
    be the set of points so that $f|_U\subset\R$. Then Lemma \ref{lem:gentf_grad_exists},
    Definition \ref{def:gradfforfLLIP} and
    Proposition
    \ref{cor:demarcated-thm}  imply that 
    $
      {\essinf}_U\sqrt{-g(\grad f,\grad f)}\geq 1.
    $
		This is in contrast to $(x,y)\mapsto h(x,y)=y$ where
    $
      {\essinf}_M\sqrt{-g(\grad h,\grad h)}=0.
    $
    Hence we have paid for a lower bound on the gradient of $f$ by letting
    $f$ diverge to $\pm\infty$ on $M$. 
      \end{Ex}
  
  In order to complete our discussion of generalised time functions, 
  we present an alternative characterisation which will play an important role later.
  
  \begin{Prop}\label{cor:demarcated-thm}
    Let $(M,g)$ be a Lorentzian manifold and $f:M\to\R$ a
    function differentiable a.e. The condition
    \begin{equation}
      \mbox{for all }x\in M,\ \mbox{for all\ }y\in I^+(x), \ \ 
      f(y)-f(x)\geq d(x,y) 
    \label{eq:flip}
    \end{equation}
    holds if and only if 
		${\essinf}_M\sqrt{-g(\grad f,\grad f)}\geq 1$ and $f$ is future-directed.
  \end{Prop}
  \begin{proof} We begin by showing that condition  \eqref{eq:flip} 
  implies that $\grad f$ is past-directed and timelike wherever it exists. 
  This allows us to build a coordinate system using $\grad f$ which is 
  then used to prove the bound on the gradient. 
%  We first prove that the condition 
%    \eqref{eq:flip} implies the bound on
%    the gradient. We are assuming that $\grad f$ exists a.e.\ and then we claim
%    that condition \eqref{eq:flip} tells us that $\grad f$ is past-directed where
%    it exists. 
%    To prove this claim, we observe that the Appendix also shows that
%    for all timelike curves $\gamma:[0,1]\to M$, the function $f\circ\gamma$ is
%    differentiable a.e. in $[0,1]$.
    We will make extensive use of the fact that for all timelike curves 
    $\gamma:[0,1]\to M$, the function $f\circ\gamma$ is
    differentiable a.e. in $[0,1]$, as shown in the Appendix.

    So we can fix $x\in M$ where $(\nabla f)(x)$ exists. Then we take a geodesic
    neighbourhood $U$ of $x$, and $y\in I^+(x)\cap U$. We let $\gamma$ be the
    unique  geodesic from $x$ to $y$ so that
    $d(x,y)=\int_0^{1}\sqrt{-g(\gamma',\gamma')}(s)ds$. Indeed, for $0< t\leq 1$,
    $d(x,\gamma(t))=\int_0^t\sqrt{-g(\gamma',\gamma')}(s)ds$.

    Now if $f$ satisfies condition \eqref{eq:flip}, then 
    $$
      f(\gamma(t))-f(x)\geq d(x,\gamma(t))=\int_0^t\sqrt{-g(\gamma',\gamma')}(s)ds.
    $$
    Dividing through by $t$ and using the mean value theorem for integrals
    shows that 
    $$
      \frac{f(\gamma(t))-f(x)}{t}
      \geq 
      \frac{1}{t}\int_0^t\sqrt{-g(\gamma',\gamma')}(s)ds
      =\sqrt{-g(\gamma',\gamma')}(t_0),
    $$
    where $0<t_0<t$. Hence as $t\to 0$ we find
    \begin{equation}
      g(\nabla f,\gamma')(\gamma(0))
      =
      \frac{d(f\circ\gamma)}{dt}(0)
      \geq
      \sqrt{-g(\gamma',\gamma')}(\gamma(0)).
      \label{eq:basic-inequality}
    \end{equation}
    By considering all such $y\in I^+(x)\cap U$,
    we see that Equation \eqref{eq:basic-inequality}
    holds for all timelike vectors in $T_xM$. In particular, letting $T$ 
    be the unit vector field defining the time orientation of $M$, 
    we find that $g(\nabla f,T)\geq 1$
    and hence $\nabla f$ is past-directed.

    If $Z\in T_xM$ is a timelike vector, and future directed, we can write
    $$
      Z=\alpha T+\beta V,\quad V\perp T, 
      \quad g(V,V)=1,
      \quad \alpha>0,
      \quad -\alpha^2+\beta^2=-m^2.
    $$
    Similarly
    $$
      (\nabla f)(x)=\mu T+\nu W,\quad W\perp T, \quad g(W,W)=1,\quad \mu\leq -1
    $$
    where the value of $\mu$ follows from setting $\gamma'(0)$ equal to $T(x)$ in 
    Equation \eqref{eq:basic-inequality}. We can, and do, assume that
    $\beta,\,\nu>0$.
    We set $c=g(V,W)$ and compute that
    $$
      g((\nabla f)(x),Z)=-\mu\alpha+\nu\beta c\geq \sqrt{\alpha^2-\beta^2}=m
    $$
    where the inequality is from Equation \eqref{eq:basic-inequality}. Now
    choose $V=-W$ so
    that $c=-1$. This yields
    $$
      |\mu|\alpha\geq m+\nu\beta .
    $$
%and so 
%$$
%|\mu|\geq m^2/\alpha+|\nu|\frac{|\beta|}{\alpha}
%=\frac{m^2}{\sqrt{m^2+\beta^2}}+\frac{|\nu||\beta|}{\sqrt{m^2+\beta^2}}.
%$$
  Rearranging and using the binomial series yields
  \begin{align*}
    |\mu|&\geq \frac{m}{\alpha}+\nu\sqrt{1-\frac{m^2}{\alpha^2}}\\
    &= \frac{m}{\alpha} +\nu\left(1-\frac{m^2}{2\alpha^2}-\frac{1}{2\times 4}\left(\frac{m^2}{\alpha^2}\right)^2+\cdots\right)\\
    &=\nu+\frac{m}{\alpha}-\frac{m^2\nu}{2\alpha^2}-\frac{\nu}{2\times 4}\frac{m^4}{\alpha^4}+\cdots.
  \end{align*}
  This makes sense as an infinite series, since $m^2/\alpha^2<1$ when $\beta\neq 0$, 
  and for $m$ sufficiently small it is 
  straightforward to see that we have $|\mu|>\nu$. In short,  $(\nabla f)(x)$
  is timelike.
    
%    To see this, first since condition \eqref{eq:flip} ensures that $f$ is increasing on every timelike
%    curve $\gamma:[0,1]\to M$, the discussion in the Appendix tells us that 
%    $f\circ\gamma$ is differentiable
%    almost everywhere. Suppose, for a contradiction, that 
%    $g(\nabla f,\gamma')<\sqrt{-g(\gamma',\gamma')}$ on a set of non-zero measure. Then there
%    is a subinterval $[a,b]\subset [0,1]$ where $g(\nabla f,\gamma')<\sqrt{-g(\gamma',\gamma')}$ almost
%    everywhere on $[a,b]$. Then 
%    $$
%    d(\gamma(b),\gamma(a))\geq\int_a^b\sqrt{-g(\gamma',\gamma')}dt>\int_a^bg(\nabla f,\gamma')dt
%    $$
%    let $T$ be the vector field defining the time orientation of $M$: without loss of generality
%    $g(T,T)=-1$.  
    
    We now know that condition \eqref{eq:flip} implies that
    $(\grad f)(x)$ is past-directed and timelike. Hence we may
    take normal 
    coordinates $\phi:U\subset \R^n\to V\subset M$ about $x$ 
    so that
    $g(\partial_i,\partial_j)(x)=\delta_{ij}$, $i\neq 0$, and 
    $g(\partial_0,\partial_j)(x)=-\delta_{0j}$
    where $\partial_0(x)=\alpha\grad f(x)$, $\alpha\neq0$. 
    This ensures that $\partial_if|_x=0$ if $i\neq 0$.
    
    Condition \eqref{eq:flip} tells us that
%    $
%      f(y)-f(x)\geq d(x,y),
%    $
%    for all $x\in M$ and $y\in \Future{x}$. 
%    Hence
    \begin{equation}
      \lim_{h\to 0^+}\frac{f\circ\phi(h,0,\ldots,0)-f\circ\phi(0,0,\ldots,0)}{h}
        \geq \lim_{h\to 0^+}\frac{d\left(x,\phi(h,0,\ldots,0)\right)}{h}
    \label{eq:inequal-1}
    \end{equation}
    and
    \begin{equation}
      \lim_{h\to 0^-}\frac{f\circ\phi(0,0,\ldots,0)-f\circ\phi(h,0,\ldots,0)}{-h}
        \geq \lim_{h\to 0^-}\frac{d\left(\phi(h,0,\ldots,0),x\right)}{-h}.
    \label{eq:inequal-2}
    \end{equation}

    By construction, for $h$ small enough, we have
    \begin{align*}
      \phi\left(h,0,\ldots,0\right)=\exp_x\left(h\partial_0\right)=
        \gamma_{h\partial_0}(1)=\gamma_{\partial_0}(h),
    \end{align*}
    where $\gamma_v:[0,a)\to M$, $a\in\R\cup\{\infty\}$ 
    is the unique geodesic satisfying
    $\gamma_v(0)=x$ and $\gamma_v'(0)=v$ with affine parameter. Since 
    $\gamma_{\partial_0}$ is a geodesic
    and $g(\gamma_{\partial_0},\gamma_{\partial_0})(x)=-1$, we see that 
    for all $0\leq \tau\leq h$, 
    $g(\gamma_{\partial_0},\gamma_{\partial_0})(\tau)=-1$.
    Then for $h>0$ we calculate that
    \begin{align*}
      L\left(\left.\gamma_{\partial_0}\right|_{[0,h]}\right)
      =\int_0^h\sqrt{-g(\gamma_{\partial_0},\gamma_{\partial_0})(\tau)}d\tau
      =\int_0^hd\tau=h.
    \end{align*}
    By definition, 
    $d(x,\phi(h,0,\ldots,0))
    =\sup_{\gamma\in\Omega_{x,\phi(h,0,\ldots,0)}}L(\gamma)$
    and as 
    $\left.\gamma_{\partial_0}\right|_{[0,h]}\in\Omega_{x,\phi(h,0,\ldots,0)}$
    we  see that
    $d(x,\phi(h,0,\ldots,0))\geq h.$ 
    The same calculation
    when $h<0$ gives $d(\phi(h,0,\ldots,0),x)\geq -h$.
    Plugging these inequalities into Equations \eqref{eq:inequal-1} 
    and \eqref{eq:inequal-2}  we see that
    $
      \left.\partial_0f\right|_x\geq 1.
    $
    We may now calculate that
    \begin{align*}
      \sqrt{-g(\grad f,\grad f)}(x)&=
        \sqrt{-g^{ij}\partial_if\partial_j f}(x)
        =\lvert\left.\partial_0f\right\rvert\hspace{-2pt}(x)
        \geq 1.
    \end{align*}
    As $x$ was an arbitrary point where the gradient exists, we find that
    \[
      \underset{M}{\essinf}\sqrt{-g(\grad f,\grad f)}\geq 1.
    \]

    For the converse statement, suppose that 
    ${\essinf}_M\sqrt{-g(\grad f,\grad f)}\geq 1$. 
    Let $x,y\in M$ and $\gamma$ be a timelike curve from $x$ to $y$.
    From Lemma \ref{Lem bound} we know that
    \[
      \lvert f(y)-f(x)\rvert=f(y)-f(x)\geq \underset{\gamma}{\essinf}\sqrt{-g(\grad f,\grad f)}L(\gamma)\geq L(\gamma).
    \]
    Since $y\in I^+(x)$ we know that $d(x,y)$ is the supremum of $L(\gamma)$ over all
    timelike curves from $x$ to $y$. Hence by
    taking the supremum over all timelike curves from $x$ to $y$ of the 
    inequality above we get
    \[
      f(y)-f(x)\geq d(x,y).
    \]
    Therefore condition \eqref{eq:flip} is satisfied by $f$.
  \end{proof}

\section{Constructions and definitions for the proof of the main theorems}
\label{sec:guts}

  This section contains the technical details for the 
  proofs of our main theorems.
  We have divided the work into three portions each of which
  culminates in a key result. Briefly those results are:
   \begin{enumerate}
    \item Lemma
      \ref{lem:existence_of_hattings}: 
      If the Lorentzian distance is finite then there exists
      a special achronal subset (which we call a hatting);
    \item Lemma
      \ref{lem:achronal_surface_for_time_function}:
      If there exists a future set with non-empty boundary then
      there exists an achronal surface $S$
      so that $M=\Future{S}\cup S\cup \Past{S}$;
    \item Proposition
      \ref{lem:thefunction}:
      If the Lorentzian distance is finite then there
      exists a generalised time function
      satisfying condition
      \eqref{eq:flip}.
  \end{enumerate}

  \subsection{Finite Lorentzian distance implies the existence of a
    hatting}
    
    \begin{Def}
      Let $M$ be a manifold. A sequence
      $(x_i)_{i\in\N}\subset M$
      such that there exists $x\in M$ such that
      $d(x_i,x)\to\infty$ ($d(x, x_i)\to\infty$)
      as $i\to\infty$
      is called future (past) divergent.
      Given a future (past) divergent
      sequence, $(x_i)$, let 
      $F_{(x_i)}=\Future{\{x\in M:\lim_{i\to\infty}d(x_i,x)=\infty\}}$
      ($P_{(x_i)}=\Past{\{x\in M:\lim_{i\to\infty}d(x,x_i)=\infty\}}$).
    \end{Def}

    \begin{Lem}\label{lem:divergentsunseqeunces}
      Let $(M,g)$ be a Lorentzian manifold and
      $(x_i)$ a future divergent sequence.
      If $(y_i)$ is a subsequence of $(x_i)$ then
      $(y_i)$ is future divergent and
      $F_{(x_i)}\subset F_{(y_i)}$.
    \end{Lem}
    \begin{proof}
      Let $x\in F_{(x_i)}$. By definition
      $\lim_{i\to\infty}d(x_i,x)=\infty$. 
      If $\lim_{i\to\infty}d(y_i,x)\neq\infty$ then we also have
      $\lim_{i\to\infty}d(x_i, x)\neq\infty$. Therefore
      $\lim_{i\to\infty}d(y_i,x)=\infty$. This implies that
      $(y_i)$ is future divergent
      and that $F_{(x_i)}\subset F_{(y_i)}$.
    \end{proof}

    \begin{Lem}\label{lem:infinitedistanceequalsdivergence}
      Let $(M,g)$ be a Lorentzian manifold with finite Lorentzian distance and $S$ an achronal
      set.
      If there exists $x\in\Future{S}$ so that $d(S,x)=\infty$
      then there exists a future 
      divergent sequence in $S$.
    \end{Lem}
    \begin{proof}
      Since $d(S,x)=\infty$ there exists a sequence $(x_i)\subset S$
      so that $\lim_{i\to\infty}d(x_i,x)=\infty$.
      The sequence $(x_i)$ is trivially a future 
      divergent sequence in $S$.
    \end{proof}

    \begin{Def}
      A 
      hatting is an achronal subset $H\subset M$ so that
      for every future (past) divergent
      sequence, $(x_i)_{i\in\N}$ in $M$,
      there exists $N\in\N$ so that for all $j\geq N$,
      $x_j\in\Past{H}$ ($x_j\in\Future{H}$).
%
%      Let $L$ be a set of future and/or past 
%      divergent sequences in $M$.
%      A 
%      hatting for $L$ is an achronal subset $H\subset M$ so that
%      for every future (past) divergent
%      sequence, $(x_i)_{i\in\N}$ in $L$,
%      there exists $N\in\N$ so that for all $j\geq N$,
%      $x_j\in\Past{H}$ ($x_j\in\Future{H}$).
%      A hatting for $M$ is a set that is a hatting for the set of all
%      divergent sequences in $M$.
    \end{Def}

    \begin{Lem}\label{lem:thereexistsanN...}
      Let $(M,g)$ be a Lorentzian manifold with finite Lorentzian distance.
      If $S\subset M$ is finite then for all future 
      divergent sequences $(x_i)$ there exists
      $N\in\N$ so that for all $j\geq N$, $x_j\not\in\overline{\Future{S}}$.
    \end{Lem}
    \begin{proof}
      For a contradiction we will assume that no such $N$ exists.
      By definition and as $S$ is finite
      $\overline{\Future{S}}=\bigcup_{s\in S}\overline{\Future{s}}$.
      As no such $N$ exists and as the union is over a finite number
      of elements a pigeon hole argument shows for all
      divergent sequences, $(x_i)$, there
      exists
      a subsequence, $(y_i)$, of $(x_i)$ so that
      $(y_i)\subset\overline{\Future{s}}$ for some $s\in S$.
      Lemma \ref{lem:divergentsunseqeunces} implies that
      $(y_i)$ is divergent. Let $y\in F_{(y_i)}$.
      By construction, for each $i$, $s\in\Past{y_i}$ and
      $y_i\in\Past{y}$. Thus $\lim_{i\to\infty}d(y_i,y)=\infty$
      implies that $d(s,y)=\infty$.
      This is a contradiction, hence the required $N\in\N$ exists.
    \end{proof}

    \begin{Lem}\label{lem:intersectionimplieshatting}
      Let $(M,g)$ be a Lorentzian manifold. Let $S\subset M$
      and let $(x_i)$ be a future 
      divergent sequence.
      If $F_{(x_i)}\cap S\neq\EmptySet$ then
      there exists $N\in\N$ so that for all $j\geq N$,
      $x_j\in\Past{S}$.
    \end{Lem}
    \begin{proof}
      Let $s\in F_{(x_i)}\cap S$. Since $F_{(x_i)}$ is open and non-empty, there exists
      $x\in F_{(x_i)}\cap \Past{s}$. Hence $x\in\Past{S}$.
      As $x\in F_{(x_i)}$ we know that $\lim_{i\to\infty}d(x_i,x)=\infty$.
      This implies that there exists $N\in\N$ so that
      for all $j\geq N$, $x_i\in\Past{x}$.
      Since $\Past{x}\subset\Past{S}$ we have the result.
    \end{proof}
    
    \begin{Lem}\label{lem:existence_of_hattings}
      Let $(M,g)$ be a Lorentzian manifold.
      If the Lorentzian distance is finite
      then there exists a hatting for $M$.
    \end{Lem}
    \begin{proof}
      Let $A$ be the union of all $F_{(x_i)}$ for all
      future divergent sequences.
      By construction $A$ is an open
      manifold. Therefore there
      exists a countable dense subset $F$ of $A$.
      Similarly, let $P$ be a countable dense subset 
      of the union of all $P_{(x_i)}$
      where $(x_i)$ is a past 
      divergent sequence.
      Since $F$ and $P$ are countable, we choose an ordering
      so
      that $F=\{f_0,f_1,\ldots, f_i,\ldots\}$,
      $P=\{p_0,p_1,\ldots,p_i,\ldots\}$. 

      We build our hatting by iteration over $\N$.\\
      \textbf{Base case:} Let $i = 0$ and define $S_0=P_0=F_0=\{p_0\}$.
      Note that $P_0$ and $F_0$ are finite,
      $\Past{S_0}\subset\Past{P_0}$ and $\Future{S_0}\subset\Future{F_0}$.
      Since the Lorentzian distance is finite the sets
      $S_0,P_0$ and $F_0$ are achronal.\\
      \textbf{Inductive case:} 
      Assume that $S_{i-1}, P_{i-1}$ and $F_{i-1}$ exist 
      and are such that $S_{i-1}$ is achronal,
      $P_{i-1}, F_{i-1}$ are finite, $\Past{S_{i-1}}\subset\Past{P_{i-1}}$
      and $\Future{S_{i-1}}\subset\Future{F_{i-1}}$.
      
      Assume that $i = 2k+1$ for some $k\in\N$, $k\geq 0$.
      If there does not exist $f_k\in F$ then let
      $S_i=S_{i-1}$, $P_i = P_{i-1}$, $F_i=F_{i-1}$
      and continue the induction.
      Otherwise, we have three subcases:
      \begin{enumerate}
        \item If $\{f_k\}\cup S_{i-1}$ is achronal let $S_i = \{f_k\}\cup
          S_{i-1}$, $P_i = \{f_k\}\cup P_{i-1}$ and $F_{i}=\{f_k\}\cup S_{i-1}$.
          It is clear that the inductive hypothesis remains true.

        \item If $f_k\in\Past{S_{i-1}}$ then, by construction, $f_k\in F$
          hence there exists $(x_i)$ a future 
          divergent sequence so that $f_k\in F_{(x_i)}$.
          Since $\Past{S_{i-1}}$ is a past set, $F_{(x_i)}$
          is a future set and as $f_k\in\Past{S_{i-1}}\cap F_{(x_i)}$
          there exists $\hat{f_k}\in\Future{f_k}\cap\partial\Past{S_{i-1}}$.
          Let $S_i=\{\hat{f}_k\}\cup S_{i-1}$. By construction $S_i$ is achronal.
          Let $P_i=P_{i-1}$. Since $\hat{f}_k\in\partial\Past{S_i}$,
          $\Past{S_i}\subset \Past{P_i}$. Let
          $F_i = \{f_k\}\cup F_{i-1}$. By construction
          $\Future{S_{i}}\subset \Future{F_i}$. 
          Hence the 
          inductive hypotheses are true.

       \item Otherwise $f_k\in\Future{S_{i-1}}$.
          Let $(x_i)$ be a future divergent sequence
          so that $f_k\in F_{(x_i)}$. Since $F_{i-1}$ is finite and
          as the Lorentzian distance is assumed to be finite, Lemma
          \ref{lem:thereexistsanN...} implies that
          there exists $N\in\N$ so that for all $j\geq N$,
          $x_j\not\in\Future{F_{i-1}}$. Since $\Future{S_{i-1}}\subset
            \Future{F_{i-1}}$ this implies that
          for all $j\geq N$ there exists 
          $y_j^{(x_i)}\in \Past{f_k}\cap\Future{x_j}\cap
            \partial\Future{S_{i-1}}$.
          Let $Y(f_k)$ be the union of all 
          $y_j^{(x_i)}$ for all $(x_i)$, a future 
          divergent sequence, so that
          $f_k\in F_{(x_i)}$.

          By construction $Y(f_k)\subset\partial\Future{S_{i-1}}$ so
          $S_i = Y(f_k)\cup S_{i-1}$ is achronal.
          Let $F_i = F_{i-1}$. Since $Y(f_k)\subset\partial\Future{S_{i-1}}$,
          $\Future{F_i}\supset\Future{S_{i}}$.
          Let $P_i = \{f_k\}\cup P_{i-1}$. Since $Y(f_k)\subset\Past{f_k}$,
          $\Past{S_i}\subset \Past{P_i}$.
          It is clear that the inductive hypothesis are satisfied.
      \end{enumerate}

      Assume that $i = 2k$ for some $k\geq 1$. This is the time reversed version of the
      three subcases above. For clarity we write them out in full.
      If there does not exist $p_k\in P$ then let
      $S_i=S_{i-1}$, $P_i = P_{i-1}$, $F_i=F_{i-1}$
      and continue the induction. Otherwise, we have three
      subcases:
      \begin{enumerate}
        \item If $\{p_k\}\cup S_{i-1}$ is achronal then let
          $S_i = \{p_k\}\cup S_{i-1}$, $P_i = \{p_k\}\cup P_{i-1}$
          and $F_i=\{p_k\}\cup F_{i-1}$.
          It is clear that the inductive hypothesis are satisfied.
        \item If $p_k\in\Future{S_{i-1}}$ then, by construction, there exists
          $(x_i)$ a past divergent sequence
          so that $p_k\in P_{(x_i)}$. Since $\Future{S_{i-1}}$ is a future
          set and $P_{(x_i)}$ is a past set and as $p_k\in\Future{S_{i-1}}\cap
          P_{(x_i)}$ there exists $\hat{p}_k\in\Past{p_k}\cap
            \partial\Future{S_{i-1}}$. Let $S_i = \{\hat{p}_k\}\cup S_{i-1}$.
          Let $F_{i} = F_{i-1}$ and $P_{i}=\{\hat{p}_k\}\cup P_{i-1}$.
          Since $\hat{p}_k\in\partial\Future{S_{i-1}}$, 
          $\Future{S_i}\subset\Future{F_i}$ and it is clear that
          $\Past{S_i}\subset\Past{P_{i}}$. Hence the inductive hypotheses
          are satisfied.
        \item Otherwise $p_k\in\Past{S_{i-1}}$.
          Let $(x_i)$ be a past divergent
          sequence so that $p_k\in P_{(x_i)}$. 
          Since $P_{i-1}$ is finite and as the Lorentzian distance
          is assumed to be finite, the time reverse
          of Lemma \ref{lem:thereexistsanN...} implies
          there exists $N\in\N$ so that
          for all $j\geq N$, $x_j\not\in\Past{P_{i-1}}$.
          Since $\Past{S_{i-1}}\subset\Past{P_{i-1}}$
          this implies that for all $j\geq N$ there exists
          $y_j^{(x_i)}\in\Future{p_k}\cap\Past{x_j}\cap\partial\Past{S_{i-1}}$.
          Let $Y(p_k)$ be the set of all $y_j^{(x_i)}$ for 
          all $(x_i)$, a past 
          divergent sequence, so that $p_k\in P_{(x_i)}$.
          
          By construction $Y(p_k)\subset\partial\Past{S_{i-1}}$
          so
          $S_i = Y(p_k)\cup S_{i-1}$ is achronal.
          Let $P_i=P_{i-1}$. As $Y(p_k)\subset\partial\Past{S_{i-1}}$ it is
          clear that $\Past{S_i}\subset\Past{P_i}$.
          Let $F_i = \{p_k\}\cup F_{i-1}$. Since $Y(p_k)\subset \Future{p_k}$
          we have that $\Future{S_i}\subset\Future{F_i}$. Hence the inductive
          hypotheses are satisfied.
      \end{enumerate}

      Since $S_i\subset S_{i+1}$ and each $S_i$ is achronal the set
      $H=\bigcup_iS_i$ is achronal. 

      We now prove that for each
      $(x_i)$, a future divergent
      sequence, there exists $N\in\N$ so that
      for all $\ell\geq N$, $x_\ell\in\Past{H}$.
      By construction there exists $f_k\in F\cap F_{(x_i)}$.
      We have three cases ($j=2k+1$):
      \begin{enumerate}
        \item If $\{f_k\}\cup S_{j}$ is achronal then $f\in S_{j+1}\subset H$.
          Lemma \ref{lem:intersectionimplieshatting} now gives
          the required $N$.
        \item If $f_k\in\Past{S_{j}}$ then, by construction there
          exists $\hat{f}_k\in\Future{f_k}\cap S_{j+1}\subset H$.
          Lemma \ref{lem:intersectionimplieshatting} now gives
          the required $N$.
        \item Otherwise $f_k\in\Future{S_{j}}$. By construction there
          exists $N\in\N$ so that for all $j\geq N$ there
          is $y_j^{(x_i)}\in S_{j+1}\cap\Future{x_j}$.
          That is for all $\ell\geq N$, $x_\ell\in\Past{H}$.
      \end{enumerate}
      The time reverse of this argument shows that for all
      past divergent sequences, $(x_i)$,
      there exists $N$ so that for all $\ell\geq N$,
      $x_\ell\in\Future{H}$.
      Hence $H$ is a hatting.
    \end{proof}

    Every generalised time function 
    satisfying condition \eqref{eq:flip} gives rise to a hatting.
    
    \begin{Prop}
    \label{prop:existenceofgtimefunction}
      Let $(M,g)$ be a Lorentzian manifold.
      If $f:M\to\R$ is a generalised time function
      that satisfies condition \eqref{eq:flip} 
      then for all $r\in \R$, in the range of $f$,
      $f^{-1}(r)$ is a hatting for $M$.
    \end{Prop}
    \begin{proof}
      Let $(x_i)_{i\in\N}$ be a future
      divergent sequence
      in $M$. If $(x_i)\subset \overline{\Future{f^{-1}(r)}}$ then
      for all $y\in F_{(x_i)}$
      $d(f^{-1}(r), y)=\infty$. This implies, by
      Lemma \ref{lem.inequalities},
      that $f(y)=\infty$
      which is a contradiction. Therefore there exists 
      $N\in\N$ so that for all 
      $i\geq N$, $x_i\in\Past{f^{-1}(r)}$.
      A similar argument applied to
      past divergent
      sequences shows that $f^{-1}(r)$ is a hatting.
    \end{proof}

\subsection{The construction of a special achronal surface}
  \label{subsec:surface} 
  This subsection shows how to construct an achronal surface, $S$,
  so that $M=I^+(S)\cup S\cup I^-(S)$.   

  \begin{Lem}\label{def:achonral_and_alexandrov_sets}
    Let $(M,g)$ be a Lorentzian manifold.
    Let $S$ be an achronal surface so that
    $S=\partial I^+(S)$ and let $x,\,y\in M$ be such that $y\in I^+(x)$ and 
    $U=I^-(y)\cap I^+(x)$. If $I^+(S)\cap U\neq\varnothing$
    then $U = \left(I^+(S)\cap U\right) \cup \left(S\cap U\right) \cup
    \left(I^-(S)\cap U\right)$.
  \end{Lem}
  \begin{proof}
		It is clear that $\left(I^+(S)\cap U\right) \cup 
    \left(S\cap U\right) \cup
    \left(I^-(S)\cap U\right)\subset U$.
    If $x\in \overline{I^+(S)}$ then the achronality of
    $S$ 
    implies that $I^+(x)\subset I^+(S)$. Hence,
    $U\subset I^+(S)$ and we have the result.
    Thus, we assume that $x\not\in\overline{I^+(S)}$. 
    
    Let $w\in U$. By construction there exists a timelike curve $\gamma$
    from $x$ to $y$ through $w$. As $I^+(S)\cap U\neq\varnothing$ 
    and $U\subset I^-(y)$
    we know that $y\in I^+(S)$.
    By assumption $x\not\in\overline{I^+(S)}$, hence
    Corollary \ref{cor.penrose3.15} 
    implies that $\gamma\cap S=\{s\}$ for a unique $s\in S$.
    %{\bf what does 3.15 say that implies this?}. 
    Since $\gamma$ is timelike,
    this implies that $w$ lies in $I^+(s)$, $I^-(s)$, or $w=s$.
    Hence $w\in I^+(S)$, $w\in S$ or $w\in I^-(S)$, and
    $U \subset \left(I^+(S)\cap U\right) \cup \left(S\cap U\right) \cup
    \left(I^-(S)\cap U\right)$ as required.
  \end{proof}
  
  As with many results in Lorentzian geometry a `time reversed' version of this
  result also holds.

  \begin{Lem}\label{def:achonral_and_alexandrov_sets_r}
    Let $(M,g)$ be a Lorentzian manifold.
    Let $S$ be an achronal surface
    so that $S=\partial I^-(S)$ and let $x,y\in M$ so that $y\in I^+(x)$ and 
    $U=I^-(y)\cap I^+(x)$. If $I^-(S)\cap U\neq\varnothing$
    then $U = \left(I^+(S)\cap U\right) \cup \left(S\cap U\right) \cup
    \left(I^-(S)\cap U\right)$.
  \end{Lem}

  \begin{Lem}\label{lem:achronal_surface_for_time_function}
    Let $(M,g)$ be a Lorentzian manifold.
    If there exists a future set, $F$, so that
    $\partial F \neq \varnothing$
    then there exists an achronal set $S$ so that
    $\partial F\subset S$ and 
    $M=I^+(S)\cup S\cup I^-(S)$.
  \end{Lem}
  \begin{proof}
    Let $S_0=\partial F$.
    By assumption $S_0\neq\varnothing$.
    By Lemma \ref{lem.penrose3.15}, 
    $\partial I^+(S_0)=\partial I^+(\partial F) = \partial I^+(\partial I^+(F))
      =\partial I^+(F)=\partial F$, since, by definition
    $I^+(F)=F$. 

    Assuming that $S_i$ is given we construct
    $S_{i+1}$ as follows. 
    \begin{equation}
      S_{i+1}=\left\{\begin{array}{ll} 
      \partial I^-(S_i) & i+1 \mbox{ odd}\\
      \partial I^+(S_i) & i+1 \mbox{ even.}
      \end{array}\right.
    \end{equation}
    Assuming $i+1$ is even then,
    by construction and Lemma
    \ref{lem.penrose3.15}, 
    $\partial I^+(S_{i+1})=\partial I^+(\partial I^+(S_i))=
    \partial I^+(S_i) = S_{i+1}$. Hence Corollary 
    \ref{cor.penrose3.15} can be applied to $S_{i+1}$.
    In the case that $i+1$ is odd the same argument shows that
    $S_{i+1}=\partial I^-(S_{i+1})$ and the time reverse of
    Corollary \ref{cor.penrose3.15} can be applied.

    Let $S=\bigcup_i S_i$. Thus $\partial F = S_0\subset S$ as required.
    We now prove that, for all $i\in\N$, $S_{i}\subset S_{i+1}$.
    Let $y\in S_{i}$. There are two cases to consider.

    \textbf{Case one:} Assume that $i+1$ is odd. 
    Since $y\in S_i$, $y\in \overline{I^-(S_i)}$.
    The achronality of $S_i$ implies that $y\not\in I^-(S_i)$
    and therefore $y\in\partial I^-(S_i)=S_{i+1}$.

    \textbf{Case two:} Assume that $i+1$ is even. 
    Since $y\in S_i$, $y\in \overline{I^+(S_i)}$.
    The achronality of $S_i$ implies that $y\not\in I^+(S_i)$
    and therefore $y\in\partial I^+(S_i)=S_{i+1}$.
    
    Thus $S_i\subset S_{i+1}$ as required.
    
    We now show that for all $x\in M$ there exists $i\in \N$ so that
    $x\in I^+(S_i)\cup S_i\cup I^-(S_i)$. Since $M$ is path connected
    there exists a curve $\gamma:[0,1]\to M$ from $\gamma(0)\in I^+(S_0)$ to
    $\gamma(1)=x$. For each $y\in \gamma$ choose $z_y\in I^+(y)$    
    and $w_y\in I^-(y)$.
 
    The set $\{\Past{z_y}\cap \Future{w_y}:y\in\gamma\}$
    is an open cover of $\gamma$. As $\gamma$ is compact there exists
    a finite open subcover, 
    $\mathcal{C}=\{I^-(z_i)\cap I^+(w_i): i = 0,\ldots, m\}$.
    
    By relabelling, if necessary, we take 
    $U_0=I^-(z_0)\cap I^+(w_0)\in\mathcal{C}$
    so that $\gamma(0)\in U_0$. Lemma \ref{def:achonral_and_alexandrov_sets}
    implies that 
    $U_0 = (I^+(S_0)\cap U_0) \cup (S_0\cap U_0) \cup (I^-(S_0)\cap U_0)$.
    Let $\gamma_0$ be the connected component of $\gamma$ in $U_0$
    containing $\gamma(0)$. Define 
    $ 
      t_1=\sup\{t\in[0,1]: \gamma_0(t)\in U_0\}.
    $
	  Again by relabelling, if necessary, we take 
    $U_1=I^-(z_1)\cap I^+(w_1)\in\mathcal{C}$
	  so that $\gamma(t_1)\in U_1$. 
	
	  We will show that either 
    $U_1=(I^+(S_0)\cap U_1) \cup (S_0\cap U_1) \cup (I^-(S_0)\cap U_1)$
	  or $U_1=(I^+(S_1)\cap U_1) \cup (S_1\cap U_1) \cup (I^-(S_1)\cap U_1)$.

	  By definition of $t_1$, and as $U_1$ is open, there exists 
	  $\epsilon>0$ so that $\gamma(t_1-\epsilon)\in U_1\cap U_0$. Thus
	  there exists $x_1\in U_0\cap U_1$. From above
	  we know that 
    $x_1\in(I^+(S_0)\cap U_0) \cup (S_0\cap U_0) \cup (I^-(S_0)\cap U_0)$.
	  We have three cases to consider.
	  
	  \textbf{Case one:} If $x_1\in I^+(S_0)$ then Lemma 
    \ref{def:achonral_and_alexandrov_sets}
	  implies that 
    $U_1=(I^+(S_0)\cap U_0) \cup (S_0\cap U_0) \cup (I^-(S_0)\cap U_0)$.
	  
	  \textbf{Case two:} If $x_1\in S_0$ then as $U_0\cap U_1$ is open 
	  there exists $x'_1\in U_0\cap U_1\cap I^+(x_1)$.
	  Since $x'_1\in I^+(x_1)$ and $x_1\in S_0$ we know that $x'_1\in I^+(S_0)$.
	  Lemma \ref{def:achonral_and_alexandrov_sets} implies that
	  $U_1=(I^+(S_0)\cap U_1) \cup (S_0\cap U_1) \cup (I^-(S_0)\cap U_1)$.
	  
	  \textbf{Case three:} Suppose that $x_1\in I^-(S_0)$. 
    Then as $I^-(S_1)=I^-(\partial I^-(S_0))$, 
    $x_1\in\Past{S_1}$ so that $\Past{S_1}\cap U_1\neq\EmptySet$.
	  Lemma \ref{def:achonral_and_alexandrov_sets_r} implies that
	  $U_1=(I^+(S_1)\cap U_1) \cup (S_1\cap U_1) \cup (I^-(S_1)\cap U_1)$.
	  
	  This inductive process can be repeated. The result is that, for some
	  $i=0,\ldots,m$ with $x=\gamma(1)\in U_i\in\mathcal{C}$ we know that,
	  for some $0\leq j\leq i$,
	  $U_i = (I^+(S_j)\cap U_i) \cup (S_j\cap U_i) \cup (I^-(S_j)\cap U_i)$.
	  Hence $x\in I^+(S_j)\cup S_j \cup I^-(S_j)$ as claimed.

	  We now show that $M=I^+(S)\cup S\cup I^-(S)$. For all $x\in M$
	  there exists $j\in\N$ so that $x\in I^+(S_j)\cup S_j \cup I^-(S_j)$.
	  By definition of $S$, $I^+(S)=\bigcup_i I^+(S_i)$ and 
    $I^-(S)=\bigcup_i I^-(S_i)$.
	  Thus $x\in I^+(S)\cup S\cup I^-(S)$. Since 
    $I^+(S)\cup S\cup I^-(S)\subset M$
	  we have the required equality.
	  
    We now show that $S$ is an achronal surface by showing that $S=\partial I^+(S)$. 
    Let $x\in\partial I^+(S)$.
    From above there exists $j\in\N$ so that 
    $x\in I^+(S_j)\cup S_j\cup I^-(S_j)$.
    If $x\in I^+(S_j)$ then, from the construction of $S$, we know that
    $x\in I^+(S)$. This is a contradiction as $I^+(S)$ is open. 
    Similarly if $x\in I^-(S_j)$ we are led to a contradiction. Therefore
    $x\in S_j\subset S$, and so $\partial I^+(S)\subset S$. 
    
    Let $x\in S$. Since $I^+(x)\subset I^+(S)$, 
    we know that $x\in \overline{I^+(S)}$. If $x\in\partial I^+(S)$, we
    are done, since then $S\subset \partial I^+(S)$ by the arbitrariness of $x$.
    
    So suppose that $x\not\in\partial I^+(S)$. Then
    $x\in I^+(S)$, by the achronality of $\partial \Future{S}$. 
    From the definition of $S$ this implies that there
    exists $i\in\N$ so that $x\in I^+(S_i)$. As $x\in S$ there exists
    $j\in\N$ so that $x\in S_j$. 
    
    Suppose that
    $i\geq j$. Since $x\in S_j$ then $x\in S_i$. 
    Thus we have that
    $x\in S_i\cap I^+(S_i)$. This is a contradiction as $S_i$ is an
    achronal surface. So assume that $j>i$. Since $S_i\subset S_j$
    we know that $I^+(S_i)\subset I^+(S_j)$. Thus, again, we get the
    contradiction $x\in S_j\cap I^+(S_j)$.
    
    Therefore, we have that $x\in\partial I^+(S)$ and hence that
    $S=\partial I^+(S)$, as required.
  \end{proof}

  \subsection{Combining surfaces and hattings to get generalised
    time functions}
    \label{subsec:surf-hat}

    In this section we show how to use the hatting and the surface
    construction in the previous sections to construct
    a surface $S$ so that $d(\cdot, S)$ and $d(S, \cdot)$ 
    are finite valued and $M=\Future{S}\cup S\cup\Past{S}$.
    This allows us to construct a generalised time function
    that satisfies condition \eqref{eq:flip}.

    \begin{Lem}\label{lem:whatIsASubSetOfQ+FiniteValuedness}
      Let $(M,g)$ be a Lorentzian manifold.
      If there exists a hatting, $H$,
      then there exists an achronal set $S$ such that, $H\subset S$,
      $M=\Future{S}\cup S\cup\Past{S}$ and
      for all $x\in M$, $d(S,x)<\infty$ and $d(x, S)<\infty$.
    \end{Lem}
    \begin{proof}
      Since $\EmptySet\neq H\subset\partial\Future{H}$,
      Lemma 
      \ref{lem:achronal_surface_for_time_function}
      can be used to generate an achronal surface, $S$,
      so that $H\subset S$.
      If there exists $x\in\Future{S}$ so that $d(S,x)=\infty$
      then Lemma \ref{lem:infinitedistanceequalsdivergence}
      implies that there exists $(x_i)$, a future 
      divergent sequence, lying in $S$. Since $H$ is a hatting
      there exists $N\in\N$ so that for all $j\geq N$,
      $x_j\in\Past{H}\subset\Past{S}$. This contradicts the
      achronality of $S$. Hence for all $x\in\Future{S}$,
      $d(S,x)<\infty$.
      The time reverse of this argument proves that
      for all $x\in\Past{S}$, $d(x,S)<\infty$ as required.
    \end{proof}

    \begin{Prop}\label{lem:thefunction}
      Let $(M,g)$ be a Lorentzian manifold.
      If there exists a hatting
      then there exists an achronal surface $S$ such that
      $M=\Future{S}\cup S\cup\Past{S}$. Moreover the
      function $f:M\to\R$ defined by
      $$
      f(x)=\left\{\begin{array}{ll}  {d}(S,x) & \text{if}\ x\in I^+(S)\\
      0 & \text{if}\ x\in S\\
      -{d}(x,S) & \text{if}\ x\in I^-(S)\end{array}\right..
      $$
      is a generalised time function which satisfies condition \eqref{eq:flip}.
    \end{Prop}
    \begin{proof}
      The existence of $S$ so that $M=\Future{S}\cup S\cup\Past{S}$
      is given
      by Lemma \ref{lem:whatIsASubSetOfQ+FiniteValuedness}.

      Let $x\in M$ then as $M=I^+(S)\cup S\cup I^-(S)$ and as these sets are
      pairwise disjoint we know that $x$ belongs to one of $I^+(S)$, $S$
      or $I^-(S)$. Hence $f$ is well defined. The finiteness of $f$ follows from 
      Lemma \ref{lem:whatIsASubSetOfQ+FiniteValuedness}.
      It is clear, by definition of $d$, that $f$ is strictly monotonically
      increasing on every timelike curve.
      
      It remains to show that $f$ satisfies condition \eqref{eq:flip}.
      Let $x\in M$,
      $y\in I^+(x)$ and for a contradiction we  assume that
      $f(y)-f(x)< {d}(x,y)$.
      We have five cases to consider.
      \begin{description}
        \item[Case one, {$\mathbf{f(x)>0}$}]
          By assumption $f(y)-f(x)= {d}(S,y)- {d}(S,x)
          < {d}(x,y)$. Hence
          $ {d}(S,y)< {d}(S,x)+ {d}(x,y)$ which contradicts
          Lemma \ref{lem.inequalities}.
        \item[Case two, {$\mathbf{f(x)=0}$}]
          By assuption $f(y)-f(x)= {d}(S,y)<
           {d}(x,y)$. This contradicts
          the definition of $ {d}$.
        \item[Case three, {$\mathbf{f(x)<0, f(y)>0}$}]
          By assumption $x\in I^-(S)$ and $y\in I^+(S)$. Since $S$ is an 
          achronal surface so that
          $M=I^+(S)\cup S\cup I^-(S)$ we know that
          $S=\partial I^+(S)=\partial I^-(S)$, 
          Proposition \ref{prop.penrose3.15} now
          implies that for any $\gamma:[a,b]\to M$, $\gamma\in \Omega_{x,y}$, 
          there
          exists a unique $t\in[a,b]$ so that $\gamma\cap S=\{\gamma(t)\}$.
          Therefore
          $$
            f(y)-f(x)=d(S,y)+d(x,S)\geq 
              L(\gamma|_{[a,t]}) + L(\gamma|_{[t,b]})=L(\gamma).
          $$
          Taking the supremum over $\Omega_{x,y}$ we see that
          $
            f(y)-f(x)= d(S,y)+d(x,S)  \geq d(x,y).
          $
          This contradicts our assumption.
        \item[Case four, {$\mathbf{f(x)<0, f(y)=0}$}]
          By assumption $f(y)-f(x)= {d}(x,S)< {d}(x,y)$.
          This contradicts
          the definition of $ {d}$.
        \item[Case five, {$\mathbf{f(y)<0}$}]
          By assumption $f(y)-f(x)=- {d}(y,S)+ {d}(x,S)< {d}(x,y)$. Hence
          $ {d}(x,S)< {d}(x,y)+ {d}(y,S)$ which contradicts
          Lemma \ref{lem.inequalities}.
      \end{description}  
      Since every case ends in a contradiction we see
      that $f$ satisfies condition \eqref{eq:flip}.
    \end{proof}

    \begin{Ex} 
    Here is an example of a stably causal manifold with 
    finite but discontinuous
    Lorentzian distance function. Simply take two dimensional 
    Minkowski space, and remove
    the segment $\{(x,y):\,y=0,\ -1\leq x\leq 1\}$. 
    \end{Ex}

    We now refine Proposition \ref{lem:thefunction} under the assumption 
    that the Lorentzian distance is continuous.
    
    \begin{Cor} \label{thm:main2}
      If, in addition to the assumptions of Proposition
      \ref{lem:thefunction} the Lorentzian distance is
      continuous
      then the function defined in Proposition
      \ref{lem:thefunction} is continuous.
    \end{Cor}
    \begin{proof}
      The aim is to show that the function $f$ we have constructed is 
      both lower and upper semi-continuous, \cite[page 101]{Kel}.
      For lower semi-continuity, let $S$ be our achronal surface, and 
      observe that by the 
      lower semi-continuity of  $d$, which holds on all Lorentzian manifolds, 
      for each $s\in S$ the function $p\mapsto d(s,p)$ is lower semi-continuous.  
      Hence $p\mapsto \sup_{s\in S}d(s,p)$ is lower semi-continuous.  
      The time symmetry
      of $f$ completes the argument.

      Since for each $s\in S$, $d(s,\cdot)$ is upper 
      semi-continuous, for each $\epsilon>0$ 
      there exists a neighbourhood $U\subset M$, $p\in U$, so that
      for any $s\in S$ and $q\in U$,
      \[
        d(s,p)\geq d(s,q)-\epsilon.
      \]
      Taking the supremum over all $s\in S$ yields
      \[
        f(p)\geq f(q),
      \]
      and so if $q_i\to p$ we see that $\lim\sup f(q_i)\leq f(p)$, and
      hence $f$ is upper semi-continuous.
    \end{proof}
    
    The last corollary and the equivalence between stably causality and 
    the existence of a continuous time function suggests the 
    following conjecture: If $(M,g)$ has finite and continuous 
    Lorentzian distance then $(M,g)$ is stably causal.

    Our techniques, based on functions necessarily constant on at least some causal
    curves, seem not to be able to address this question, despite obtaining a continuous generalised 
    time function when the Lorentzian distance is finite and continuous.

%    \begin{Cor}
%      If, in addition to the assumptions of Lemma
%      \ref{lem:thefunction} the Lorentzian distance is
%      continuous and the function $f:M\to\R$ 
%      defined in Lemma 
%      \ref{lem:thefunction} is such that
%      for all causal curves $\gamma$,
%      $\esssup_\gamma g(\nabla f,\nabla f)\leq -1$ then
%      $f$ is a time function.
%    \end{Cor}
%    \begin{proof}
%      Corollary \ref{thm:main2} implies that $f$ is continuous
%      while Corollary \ref{cor:comparing_definitions_of_generalised_time_functions}
%      implies that $f$ is strictly monotonically increasing along
%      every causal curve. Thus $f$ is a time function.
%    \end{proof}
%
\section{Proof of the main results}
\label{sec:main-results}

  \begin{MainResult}
    Let $(M,g)$ be a Lorentzian manifold.
    The Lorentzian distance is finite if and only if there exists
    a generalised time function $f:M\to\R$, strictly monotonically increasing on timelike curves, 
    whose gradient
    exists almost everywhere and is such that
    $\esssup\, g(\grad f,\grad f)\leq -1$.
  \end{MainResult}
  \begin{proof}
    Suppose that such a generalised time function exists.
    Then Corollary \ref{cor:gen--fin} proves that
    the Lorentzian distance is finite.

    Conversely suppose that the Lorentzian distance 
    is finite. Then Lemma \ref{lem:existence_of_hattings}
    along with Propositions \ref{lem:thefunction} and \ref{cor:demarcated-thm}, imply that there
    exists a generalised time function $f:M\to\R$
    so that $\esssup\, g(\nabla f,\nabla f)\leq -1$.
  \end{proof}

%\section{Proof of the Lorentzian Distance Formula}
%  \label{sec:formula}
  
  \begin{DistanceFormula}\label{thm:demarcated-thm}
    Let $(M,g)$ have finite Lorentzian distance.
    Then for all $p,\,q\in M$
    \begin{align*}
    d(p,q)=\inf\left\{\max\{f(q)-f(p),0\}:\ 
        f:M\to\R,\ f\ {\rm future\ directed},\ \esssup\,g(\nabla f,\nabla f)\leq -1\right\}.
      %d(p,q)=\inf\left\{|f(q)-f(p)|:\ 
        %f:M\to\R,\ \essinf\sqrt{-g(\nabla f,\nabla f)}\geq 1\right\}.
    \end{align*}
  \end{DistanceFormula}
  \begin{proof}
    We assume that either $\{p,q\}$ is achronal
    or $q\in\Future{p}$. 

    Since the Lorentzian distance is finite, Lemma
    \ref{lem:existence_of_hattings} implies that there exists
    a hatting for $M$.
    Thus Lemma 
    \ref{lem:whatIsASubSetOfQ+FiniteValuedness}
    implies that there exists an achronal surface, $S_1$,
    so that $d(S_1,\cdot)$ and $d(\cdot, S_1)$ are finite valued.
    Let $S_2=\partial\left(\Future{S_1}\setminus\Past{q}\right)$.
    Let $x\in\Future{S_2}$.
    By construction $d(S_2, x)\leq\max\{d(S_1,x), d(q,x)\}$.
    Hence $d(S_2,\cdot)$ is finite valued.
    A similar argument shows that $d(\cdot, S_2)$ is finite
    valued.
    Let $S=\partial\left(\Past{S_2}\setminus\Future{p}\right)$.
    The same arguments as above show that
    $d(S, \cdot)$ is finite
    and $d(\cdot, S)$ is finite.

    Take $f:M\to\R$ to be as defined in Proposition
    \ref{lem:thefunction} using the surface $S$.
    If $\{p,q\}$ is achronal then $p,q\in S$ so
    that $f(q)=0$ and $f(p)=0$. In this case
    $d(p,q)=0=f(q)-f(p)$.
    Now suppose that $q\in\Future{p}$.
    Let $\gamma$ be a timelike curve from $q$ to $x\in S$.
    By construction $x\not\in S_2$. Since $\gamma$ is timelike
    $x\not\in\partial\Past{q}$. 
    Hence $x\in\partial\Future{p}$ and therefore $d(p, x)=0$.  We thus have
    that $d(p,q)\geq d(p, x) + d(x, q)=d(x,q)$ with equality for $x=p$. Noting that
    $p\in S$, by taking the supremum over all $x\in S$ we get $d(p,q)=d(S,q)=f(q)$.
    Again as $p\in S$ we have that $f(p)=0$ and hence $f(q)-f(p)=d(p,q)$.    

    Lastly Proposition \ref{cor:demarcated-thm}
    implies that as $f(q)-f(p)\geq d(p,q)$
    then
    $\essinf\sqrt{-g(\nabla f,\nabla f)}\geq 1$.

    %In the case that $p\in\Future{q}$ then as
    %$\lvert f(q)- f(p)\rvert = f(p)-f(q)$ we can apply
    %the same arguments as above to show that
    %$\lvert f(q) -  f(p)\rvert = d(q,p)$.

    Thus we have shown that
    \[
    d(p,q)=\inf\left\{\max\{f(q)-f(p),0\}:\ 
        f:M\to\R,\ f\ \mbox{future directed},\ \esssup\,g(\nabla f,\nabla f)\leq -1\right\}.
      %\inf\left\{\lvert f(q)-f(p)\rvert:\ 
        %f:M\to\R,\ \essinf\sqrt{-g(\nabla f,\nabla f)}\geq 1\right\} = d(p,q),
    \]
    as required.
  \end{proof}
  
  {\bf Remark}. In fact we have shown that the infimum is achieved.
  
	\appendix  
  
\section{The  differentiability of functions that are
monotonic on timelike curves.}
\label{sec_diff_tfs}

  We begin by addressing the question of continuity.

  \begin{Prop}\label{prop.mono_is_cont_point}
    Let $f:M\to\R$ be monotonically increasing on all timelike curves
    and 
    let $x\in M$. Suppose that
    there exists a timelike curve
    $\gamma:[-1,1]\to M$ with $\gamma(0)=x$
    such that $f\circ\gamma$ is continuous at $0$.
    Then $f$ is continuous at $x$.
  \end{Prop}
  \begin{proof}
    Let $(y_i)_{i\in\N}\subset M$ be a sequence of points
    so that $y_i\to x$. Then, for each $i\in N$ there
    exists $k_i\in N$ so that for all $j>k_i$
    \[
      y_j\in I^-\left(\gamma\left(\frac{1}{i}\right)\right)
        \cap I^+(\left(\gamma\left(-\frac{1}{i}\right)\right).
    \]
    Since $f$ is monotonically increasing on all timelike curves
    this implies that, for all $j>k_i$,
    \[
      f\left(\gamma\left(\frac{1}{i}\right)\right)\geq f(y_j)\geq
      f(\left(\gamma\left(-\frac{1}{i}\right)\right).
    \]
    As $f\circ\gamma$ is continuous at $0$,
    \[
      f\left(\gamma\left(\frac{1}{i}\right)\right)\to f\circ\gamma(0)
    \]
    and
    \[
      f\left(\gamma\left(-\frac{1}{i}\right)\right)\to f\circ(\gamma(0)).
    \]
    This implies that $f(y_i)\to f(\gamma(0))=f(x)$, as required.
  \end{proof}

  \begin{Cor}
    Let $f:M\to\R$ be monotonically increasing on all timelike curves.
    Then $f$ is continuous a.e.
  \end{Cor}
  \begin{proof}
    This follows directly from Proposition \ref{prop.mono_is_cont_point} 
    and as the 
    push forward of our measure 
    on $M$ is a product measure on the image of a chart 
    $\phi:U\subset M\to V\subset\R^{n+1}$.
  \end{proof}

  We now show that functions that are monotonic on timelike curves
  are differentiable a.e. 
  
  \begin{Lem}\label{lem:monotonic_increas_diff_time_curves}
    Let $\gamma:I\to M$ be a timelike curve and
    $f:M\to\R$ be monotonic on any timelike curve.
    Then $\gamma'(f):I\to\R$ exists a.e.
  \end{Lem}
  \begin{proof}
    By definition
    \[
      \gamma'(f)|_{\gamma(t)} = \frac{d}{d\tau}f\circ\gamma|_t.
    \]
    By assumption $f\circ \gamma$ is a monotonic function. 
    Hence, standard results, e.g. \cite[Theorem 9.3.1]{Haaser1991Real}, 
    imply that $\gamma'(f)$ exists a.e.\ on $\gamma(I)$.
  \end{proof}
  
  %This lemma implies that we can differentiate functions, $f:M\to\R$,
  %so that $\FLLip{f}\geq 1$ with respect to any vector field.
  
  \begin{Lem}\label{lem:differentiability_of_time_functions}
    Let $U$ be a coordinate neighbourhood of $M$.
    Let $f:M\to \R$ be monotonic on any timelike curve
    and let $v\in TU$ a vector field on
    $U$. Then $v(f)$ exists a.e.\ on $U$.
  \end{Lem}
  \begin{proof}
    Let $\partial_0,\ldots,\partial_n$ be the coordinate vector fields on $U$.
    By using the Gram-Schmidt process we can produce an orthonormal frame field,
    $w_0,\ldots,w_n\in TU$
    over $U$ so that for all $j=1,\ldots,n$ and $i=0,\ldots,n$
    we have that $g(w_0,w_j)=-\delta_{0j}$ and $g(w_i,w_j)=\delta_{ij}$.
    Choose $1>\epsilon>0$ and
    let $e_0=w_0$ and, for all $i=1,\ldots,n$ let $e_i=(1-\epsilon)w_i+w_0$.
    Then for all $i=1,\ldots,n$ and $j=0,\ldots,n$
    we have that $g(e_0,e_j)=-1$ and $g(e_i,e_j)=(1-\epsilon)^2\delta_{ij}-1$.
    In particular this implies that each vector $e_0,\ldots,e_n$ is 
    timelike.

    Choose $x\in U$. Then for each $j=0,\ldots,n$ there exists an integral
    curve of the vector field $e_j$ through $x$. Since each $e_j$ is
    causal this integral curve is causal and therefore
    $e_j(f)$ exists a.e.\ on every integral curve of $e_j$. As these integral 
    curves foliate $U$, we know that for each $j=0,\ldots,n$, $e_j(f)$ exists
    a.e.\ on $U$.

    Since $e_0,\ldots,e_n$ are a frame over $U$ we can express $v$
    as $v=v^je_j$. Thus for $x\in U$
    we have that $v(f)=v^je_j(f)$, if $e_j(f)$ exists for all $j=1,\ldots,n$.
    Hence $v(f)$ exists a.e.\ on $U$.
  \end{proof}

  This allows us to define  the differential of $f$.
  
  \begin{Lem}\label{lem:gentf_grad_exists}
    Let $f:M\to\R$ be monotonic on any timelike curve.
    Then there exists a unique, a.e.\ defined, 
    linear operator
    $df:TM\to\R$ so that $df(v)=v(f)$.
  \end{Lem}
  \begin{proof}
    We will define $df$ locally and then show that the definitions on each
    coordinate patch satisfy the necessary transformation properties in order 
    to conclude global existence.

    Choose a coordinate neighbourhood $U$. Choose a frame
    $e_0,\ldots, e_n$ so that for all $i=1,\ldots,n$ and $j=0,\ldots,n$
    we have that $g(e_0,e_j)=-1$ and $g(e_i,e_j)=(1-\epsilon)^2\delta_{ij}-1$ with $1>\epsilon>0$ 
    (see the proof 
    of Proposition \ref{lem:differentiability_of_time_functions} for the existence
    of such frames). For all $j=0,\ldots,n$ let $de^j$ be the differential form
    defined by, for all $k=0,\ldots,n$,  $de^j(e_k)=\delta_{k}^j$.

    Define $df:TU\to \R$ by $df=e_j(f)de^j$. Since for all $j=0,\ldots,n$ the 
    vector field
    $e_j(f)$ exists a.e.\ on $U$ we know that $df$ is defined a.e.\ on $U$.
    The linearity of $df$ follows from the linearity of each $de^j$, for
    $j=0,\ldots,n$. 

    Let $V$ be a second coordinate neighbourhood so that
    $V\cap U\neq\varnothing$. Let $e_0',\ldots,e_n'$ be a frame on $V$
    so that for all $i=1,\ldots,n$ and $j=1,\ldots,n$
    we have that $g(e_0',e_j')=-1$ and $g(e_i',e_j')=(1-\epsilon)^2\delta_{ij}-1$. 
    We can write $e_j=T_j^{i}e_i'$ on $V\cap U$. This implies that
    $e_j(f)=T_j^ie_i'(f)$ and that $de^j = \left(T^{-1}\right)^j_ide'^{i}$.
    Hence, for $i,j,k,l=0,\ldots,n$,
    \begin{align*}
      df & = e_j(f)de^j = \delta_k^j e_j(f) de^k
	        = \delta_k^jT_j^ie_i'(f)\left(T^{-1}\right)^k_lde'^{l} 
	        = \delta_k^jT_j^i\left(T^{-1}\right)^k_l e_i'(f) de'^{l}\\
	       & = T_k^i\left(T^{-1}\right)^k_l e_i'(f) de'^{l} 
	        = \delta_l^i e_i'(f) de'^{l} 
	        = e'_{l}(f) de'^{l}.
    \end{align*}
    Which is the expression for $df$ in the frame on $V$. Hence
    $df$ is globally defined. 

    It remains to show that $df(v)=v(f)$ and that $df$
    is unique. Let $v\in TU$, then we can express $v$ as $v=v^ke_k$.
    Thus $df(v)=e_j(f)de^j(v^ke_k)=v^je_j(f)=v(f)$. 

    Suppose that there exists $\omega:TM\to\R$ a linear operator so that 
    $\omega(v)=v(f)$. Then
    $\omega(e_j)=e_j(f)$ and we have $\omega=e_j(f)de^j$. This implies that 
    $\omega=df$ as required.
  \end{proof}

%  We have committed a serious abuse of notation by labelling the linear operator
%  of the proposition above as $df$. This is because the differential of a function
%  on a manifold is usually defined on $C^1$ functions, $h$. In this case 
%  the differential $dh$, defined via the equation $dh(v)=v(h)$ and expressed
%  in coordinates $\phi:U\subset M\to\R^n$, agrees with the
%  standard $n$-dimensional derivative, $L$, of the image of $h$ in the same 
%  coordinates $\phi$, defined by
%  \[
%    \lim_{x\to p}\frac{\lVert \phi\circ f(x) - \phi\circ f(p) - L(x-p)\rVert}{\lVert x-p\rVert} = 0.
%  \]
%  Since we do not require $f$ and its partial derivatives to be continuous this 
%  parallel between $df$ and $L$, with respect to some chart, no longer holds.
%
  We  can now define the gradient.
  
  \begin{Def}\label{def:gradfforfLLIP}
    Let $f:M\to\R$ be monotonic on any timelike curve. We call $df$, as given in the proposition
    above, the differential of $f$.
    We call the a.e.\ defined vector field
    $\grad f$ such that
    $df(v)=g(\grad f,v)$, for all $v\in TM$, the gradient of $f$.
  \end{Def}

  The gradient of a function that is monotonically increasing on all
  timelike curves is necessarily causal.

  \begin{Lem}\label{lem:monotonicImpliesPastDirectedAndCausal}
    If $f:M\to\R$ is monotonically
    increasing on any future-directed timelike curve, then $\grad f$
    is past-directed and causal wherever it exists.
  \end{Lem}
  \begin{proof}
    Let $x\in M$ be such that $\grad f$ exists at $x$.
    Suppose that $\grad f$ is spacelike. Then there exists $v\in T_xM$
    a future-directed 
    timelike vector so that $g(\grad f, v)=0$. Choose $a\in\R$
    so that $a < 0$ and $a^2 < -\frac{g(v,v)}{g(\grad f,\grad f)}$
    and let
    $w = a\grad f + v$. Then 
    \[
      g(w, w)= a^2g(\grad f, \grad f) + g(v,v) < 0
    \]
    so that $w$ is timelike.
    Let $\gamma:(-1,1)\to M$ be a future-directed
    timelike curve so that $x=\gamma(0)$ and
    $\gamma'(0)=w$. By definition
    \[
      \frac{d}{dt}f\circ\gamma|_{t=0}=g(\grad f, w) = a g(\grad f, \grad f)< 0.
    \]
    This contradicts the assumption that $f$ is monotonically increasing
    along any future directed timelike curve. Thus $\grad f$ is causal.
    A similar argument to that above can be used to show that $\grad f$
    is past-directed.
  \end{proof}
  
\bibliographystyle{plain}

\end{document}